\documentclass[12pt,a4paper]{amsart}
\usepackage{color,graphicx,fullpage,amssymb}
\usepackage{url}
\usepackage[colorlinks=true]{hyperref}
\newtheorem{theorem}{Theorem}[section]
\newtheorem{lemma}[theorem]{Lemma}
\newtheorem{proposition}[theorem]{Proposition}
\newtheorem{corollary}[theorem]{Corollary}

\newtheorem{maintheorem}{Theorem}

\theoremstyle{definition}
\newtheorem{definition}[theorem]{Definition}
\newtheorem{remark}[theorem]{Remark}

\newcommand{\Q}{\mathbb{Q}}
\newcommand{\Qp}{\mathbb{Q}_p}
\newcommand{\Cp}{\mathbb{C}_p}
\newcommand{\Zp}{\mathbb{Z}_p}
\newcommand{\N}{\mathbb{N}}
\newcommand{\R}{\mathbb{R}}
\newcommand{\Z}{\mathbb{Z}}

\newcommand{\dd}{\mathrm{d}}
\newcommand{\ii}{\mathrm{i}}
\newcommand{\M}{\mathcal{M}}
\newcommand{\F}{\mathcal{F}}
\newcommand{\sphere}{\mathrm{S}^2_p}

\newcommand{\tr}{^{\mathrm{T}}}
\newcommand{\ocal}{\mathcal{O}}
\newcommand{\letnpos}{Let $n$ be a positive integer}
\newcommand{\letpprime}{Let $p$ be a prime number}
\renewcommand{\le}{\leqslant}
\renewcommand{\ge}{\geqslant}
\AtBeginEnvironment{pmatrix}{\let\frac\dfrac}
\DeclareMathOperator{\ord}{ord}
\DeclareMathOperator{\DSq}{DSq}
\numberwithin{equation}{section}

\newenvironment{enumerate-roman}{\begin{enumerate}
		
	}{\end{enumerate}}

\newenvironment{enumerate-alph}{\begin{enumerate}
		
	}{\end{enumerate}}

\title{$p$-adic angular momentum coupling in symplectic geometry}
\author[Luis Crespo, \'Alvaro Pelayo]{Luis Crespo\,\,\,\,\,\, \'Alvaro Pelayo}

\address{Luis Crespo,
	Departamento de Matem\'{a}ticas, Estad\'{i}stica y Computaci\'{o}n, Universidad de Cantabria, Av.~de Los Castros 48, 39005 Santander, Spain}
\email{luis.cresporuiz@unican.es}

\address{\'Alvaro Pelayo,
	Facultad de Ciencias Matem\'aticas,
	Universidad Complutense de Madrid, 28040 Madrid, Spain, and Real Academia de Ciencias Exactas, F\'isicas y Naturales, Madrid, Spain}
\email{alvpel01@ucm.es}

\begin{document}
	
\begin{abstract}
	The coupled angular momentum is an integrable system with two degrees of freedom which is fundamental in physics and the theory of integrable systems. It is obtained by coupling two angular momenta. We construct a $p$-adic analog of this system for any prime number $p$ and describe its symplectic normal forms at the critical points. This analog has a rich singularity theory with up to thirteen non-equivalent symplectic normal forms, which stands in contrast with the real case where there are exactly three normal forms.
\end{abstract}

\maketitle

\section{Introduction}

The most interesting features of integrable systems are encoded in their critical points. The study of these points is also essential in the global symplectic classifications in terms of symplectic invariants of toric \cite{Atiyah,Delzant,GuiSte} and semitoric integrable systems \cite{PPT,PelVuN-semitoric,PelVuN-construct} as well as in the systematic study of convexity \cite{RWZ}. In the present paper we use symplectic geometry \emph{but with coefficients in the field of $p$-adic numbers $\Qp$}, where $p$ is any prime number, in order to study the critical points of the $p$-adic analog of a fundamental integrable system: the coupled angular momentum.
We recommend Lurie's lecture \cite{Lurie} and the book by Scholze-Weinstein \cite{SchWei} for introductions to $p$-adic geometry and the books by Hofer-Zehnder \cite{HofZeh} and McDuff-Salamon \cite{McDSal} for introductions to symplectic geometry.

The coupled angular momentum system consists of a pair of spin systems with an energy component which depends on a parameter $t$. This parameter intuitively specifies the degree of coupling of both systems. The same can be said in the $p$-adic case: we have two $p$-adic spin systems, which in some sense represent the $p$-adic equivalent of rotation, and which have a coupled energy in the $p$-adic space. Next we give, for any prime number $p$, a construction of the $p$-adic coupled angular momentum, and state our main theorems. These results give a full description of the symplectic singularity theory of this system.

\subsection{The $p$-adic coupled angular momentum: basic properties}

The \emph{(real) coupled angular momentum} is an integrable system depending on three parameters $t,R_1,R_2\in\R$, with $0\le t\le 1$ and $0<R_1<R_2,$ which results from coupling, in a non-trivial fashion, two angular momenta. It has been studied by Sadovskii and Zhilinskii \cite{SadZhi} and by Le Floch-Pelayo \cite{LeFPel} among other authors, and it is fundamental in physics and in symplectic geometry (see Figure \ref{fig:real-angular}). The following definition gives the construction of its $p$-adic analog, for any prime number $p$. Recall that a \emph{$p$-adic analytic integrable system} on a $p$-adic analytic symplectic manifold $(M,\omega)$ of dimension $4$ is a $p$-adic analytic map $F=(f_1,f_2):(M,\omega)\to(\Qp)^2$ such that $\{f_1,f_2\}=0$ and the set where $\dd f_1$ and $\dd f_2$ are linearly independent is a dense subset of $M$.

\begin{definition}[$p$-adic coupled angular momentum]\label{def:angular}
	\letpprime. Let $R_1,R_2\in\Qp\setminus\{0\}$. Let $\sphere$ be the $p$-adic sphere, that is,
	\begin{equation}\label{eq:sphere}
		\sphere=\Big\{(x,y,z)\in(\Qp)^3:x^2+y^2+z^2=1\Big\}.
	\end{equation}
	Let $(x_1,y_1,z_1,x_2,y_2,z_2)$ be the standard $p$-adic coordinates on the product $\sphere\times\sphere$ and let $\omega_1$ and $\omega_2$ be the standard $p$-adic analytic area forms on the first and second factor of this product, respectively, that is,
	$\omega_1=\frac{1}{z_1}\dd x_1\wedge\dd y_1(=\frac{1}{y_1}\dd z_1\wedge\dd x_1=\frac{1}{x_1}\dd y_1\wedge\dd z_1)$
	and $\omega_2=\frac{1}{z_2}\dd x_2\wedge\dd y_2$. Endow $\sphere\times\sphere$ with the symplectic form $R_1\omega_1\oplus R_2\omega_2$. Let
	\begin{equation}\label{eq:F}
		\mathcal{F}=\Big\{(t,R_1,R_2)\in\Zp\times(\Qp)^2:|R_2|_p>|R_1|_p>0\Big\}.
	\end{equation}
	Let $(t,R_1,R_2)\in\F$. The \emph{$p$-adic coupled angular momentum corresponding to the triple $(t,R_1,R_2)$} is the $p$-adic analytic integrable system $F_{t,R_1,R_2}:\sphere\times\sphere\to(\Qp)^2$ given by
	\[F_{t,R_1,R_2}(x_1,y_1,z_1,x_2,y_2,z_2)=\Big(J_{t,R_1,R_2}(x_1,y_1,z_1,x_2,y_2,z_2),H_{t,R_1,R_2}(x_1,y_1,z_1,x_2,y_2,z_2)\Big),\]
	where
	\begin{equation}\label{eq:angular}
		\left\{\begin{aligned}
			J_{t,R_1,R_2}(x_1,y_1,z_1,x_2,y_2,z_2) & =R_1z_1+R_2z_2; \\
			H_{t,R_1,R_2}(x_1,y_1,z_1,x_2,y_2,z_2) & =(1-t)z_1+t(x_1x_2+y_1y_2+z_1z_2).
		\end{aligned}\right.
	\end{equation}
	For brevity we refer to $F_{t,R_1,R_2}:\sphere\times\sphere\to(\Qp)^2$ as the \emph{$p$-adic coupled angular momentum}, without specifying the values of the parameters; with the understanding that for each $(t,R_1,R_2)$ we have a single $p$-adic analytic integrable system. We refer to Figure \ref{fig:padic-spheres} for a symbolic representation of the phase space of this system.
\end{definition}

\begin{figure}
	\includegraphics[width=0.7\linewidth,trim=5cm 3cm 5cm 3cm,clip]{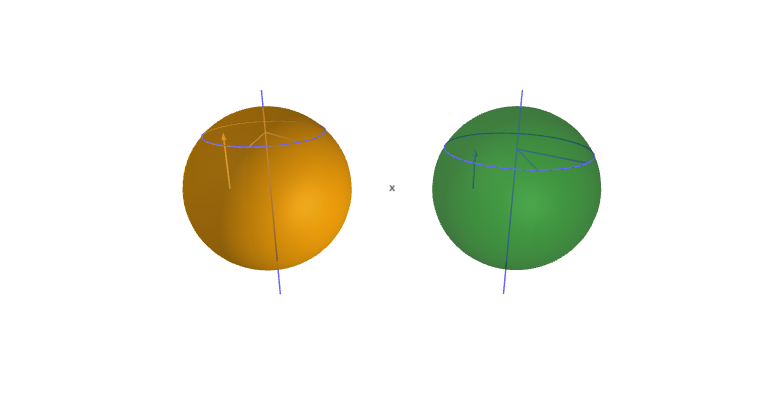}
	\caption{In the real case, the coupled angular momentum has four rank zero non-degenerate critical points corresponding to pairing the north and south poles of each of the two spheres. Three of these points are of elliptic-elliptic type, and the remaining one is of focus-focus type.}
	\label{fig:real-angular}
\end{figure}

In the theory of integrable systems (both real and $p$-adic) the most interesting points are the critical points of the system. Among them, we understand the best those which are non-degenerate in the following sense.

\begin{definition}[$p$-adic non-degenerate critical point {\cite[Definition 10.2]{CrePel-williamson}}]
	\letpprime. Let $(M,\omega)$ be a $p$-adic analytic symplectic manifold of dimension $4$, $F=(f_1,f_2):M\to(\Qp)^2$ a $p$-adic analytic integrable system and $m$ a critical point of $F$. Let $\Omega$ be the matrix of $\omega_m$. If $m$ has rank $0$, we say that $m$ is \emph{non-degenerate} if for every generic choice of $a_1,a_2\in\Qp$ the $p$-adic matrix
	$\Omega^{-1}(a_1\dd^2f_1+a_2\dd^2 f_2)$
	has $4$ different eigenvalues. If $m$ has rank $1$ and $\dd f_1=0$, we consider local linear symplectic coordinates $(x,\xi,y,\eta)$ centered at $m$ such that $\dd f_1=\dd \eta$. We say that $m$ is \emph{non-degenerate} if the first two rows and columns of $\Omega^{-1}\dd^2f_1$ form a nonsingular matrix.
\end{definition}

The following result gives a summary of the structure of the set of critical points of the $p$-adic coupled angular momentum system in \eqref{eq:angular}.

\begin{maintheorem}[Basic properties of the critical points of the $p$-adic coupled angular momentum]\label{thm:new-main}
	\letpprime. Let $\F$ be the parameter set in \eqref{eq:F} and let $(t,R_1,R_2)\in\F$. The $p$-adic coupled angular momentum $F_{t,R_1,R_2}:\sphere\times\sphere\to(\Qp)^2$ given in Definition \ref{def:angular}
	is indeed a $p$-adic analytic integrable system. For every choice of the parameters except at most two concrete values of $t$ (which depend on $R_1$ and $R_2$), every critical point of $F_{t,R_1,R_2}$ is non-degenerate. For every $(t,R_1,R_2)\in\F$, it holds that $F_{t,R_1,R_2}$ has exactly four rank $0$ critical points:
	\[\mathcal{Z}=\Big\{(0,0,1,0,0,1),(0,0,-1,0,0,1),(0,0,1,0,0,-1),(0,0,-1,0,0,-1)\}.\]
	For each $m\in\mathcal{Z}$ we define the set
	\[\mathcal{N}_m=\Big\{(g_1,g_2)\mid\exists(t,R_1,R_2)\in\F,(g_1,g_2)\text{ is local normal form of $F_{t,R_1,R_2}$ at $m$}\Big\}\]
	Let $\sim$ be the equivalence relation on $\mathcal{N}_m$ given by local symplectomorphism centered at $m$. Then the quotient set $\mathcal{N}_m/\sim$ has cardinality at most $9$. Furthermore:
	\begin{enumerate-roman}
		\item if $p=2$ then $|\mathcal{N}_m|=1$ for every $m\in\mathcal{Z}$;
		\item if $p\equiv 1\mod 4$ then for each $j\in\{1,5\}$ there is $m_j\in\mathcal{Z}$ with $|\mathcal{N}_{m_j}|=j$;
		\item if $p\equiv 3\mod 4$ then for each $j\in\{3,9\}$ there is $m_j\in\mathcal{Z}$ with $|\mathcal{N}_{m_j}|=j$.
	\end{enumerate-roman}
	Moreover, for every $(t,R_1,R_2)\in \F$, the rank $1$ critical points of $F_{t,R_1,R_2}$ are non-degenerate and have two degrees of freedom. Their images by $F_{t,R_1,R_2}$ form a curve $\gamma_{t,R_1,R_2}:\Qp\to(\Qp)^2$ which depends on $t$, $R_1$ and $R_2$. Furthermore, given any rank $1$ $p$-adic local normal form $(g_1,g_2)$, there exists a triple of parameters $(t,R_1,R_2)\in\F$ and a rank $1$ critical point of the system $F_{t,R_1,R_2}$ whose local normal form is $(g_1,g_2)$.
\end{maintheorem}

In contrast with Theorem \ref{thm:new-main}, in the real case the coupled angular momentum has four rank zero isolated singularities. Their type depends on the values of parameters but can only be of focus-focus type
$(x,\xi,y,\eta)\longmapsto(x\eta-y\xi,x\xi+y\eta)$
or of elliptic-elliptic type
$(x,\xi,y,\eta)\longmapsto\left(\frac{x^2+\xi^2}{2},\frac{y^2+\eta^2}{2}\right)$. It also has infinitely many non-isolated rank one singularities. See Figure \ref{fig:real-angular} for a representation.

Theorem \ref{thm:new-main} is a consequence of the stronger results Theorems \ref{thm:number}, \ref{thm:main}, \ref{thm:number1}, \ref{thm:main1}. In these results the cardinality of $\mathcal{N}_m$ is computed more precisely, and we give the explicit expression of the normal forms, which also makes the statements more technical.

\subsection{Classification of rank $0$ singularities}

In order to formulate our results first we need to recall the $p$-adic Weierstrass-Williamson classification in dimension $4$ \cite[Theorem A]{CrePel-williamson}. For each prime number $p$, let $c_0$ be the smallest quadratic non-residue modulo $p$. Following \cite[Definition 1.1]{CrePel-williamson}, we consider the sets:
\[X_p=\begin{cases}
	\{1,c_0,p,c_0p,(c_0)^2p,(c_0)^3p,c_0p^2\} & \text{if }p\equiv 1\mod 4; \\
	\{1,-1,p,-p,p^2\} & \text{if }p\equiv 3\mod 4; \\
	\{1,-1,2,-2,3,-3,6,-6,12,-18,24\} & \text{if }p=2.
\end{cases}\]
\[Y_p=\begin{cases}
	\{c_0,p,c_0p\} & \text{if }p\equiv 1\mod 4; \\
	\{-1,p,-p\} & \text{if }p\equiv 3\mod 4; \\
	\{-1,2,-2,3,-3,6,-6\} & \text{if }p=2,
\end{cases}\]
and the functions $\mathcal{C}_i^k:Y_p\times(\Qp)^4\to\Qp$ and $\mathcal{D}_i^k:Y_p\times(\Qp)^4\to\Qp$, for $k\in\{1,2\}$, $i\in\{0,1,2\}$, defined by:
\begin{align*}
	&\mathcal{C}_0^1(c,t_1,t_2,a,b)=\frac{ac}{2(c-b^2)},
	&&\mathcal{C}_1^1(c,t_1,t_2,a,b)=\frac{b}{b^2-c}, \\
	&\mathcal{C}_2^1(c,t_1,t_2,a,b)=\frac{1}{2a(c-b^2)},
	&&\mathcal{C}_0^2(c,t_1,t_2,a,b)=\frac{abc}{2(b^2-c)},\\
	&\mathcal{C}_1^2(c,t_1,t_2,a,b)=\frac{c}{c-b^2},
	&&\mathcal{C}_2^2(c,t_1,t_2,a,b)=\frac{b}{2a(b^2-c)},\\
	&\mathcal{D}_0^1(c,t_1,t_2,a,b)=-\frac{t_1+bt_2}{2a},
	&&\mathcal{D}_1^1(c,t_1,t_2,a,b)=-bt_1-ct_2,\\
	&\mathcal{D}_2^1(c,t_1,t_2,a,b)=-\frac{ac(t_1+bt_2)}{2},
	&&\mathcal{D}_0^2(c,t_1,t_2,a,b)=-\frac{bt_1+ct_2}{2a},\\
	&\mathcal{D}_1^2(c,t_1,t_2,a,b)=-c(t_1+bt_2),
	&&\mathcal{D}_2^2(c,t_1,t_2,a,b)=-\frac{ac(bt_1+ct_2)}{2}.
\end{align*}

\begin{definition}[Class (1), (2) or (3) normal forms]\label{def:classes}
	By \cite[Theorem A]{CrePel-williamson}, the local normal form of a $p$-adic analytic integrable system on a $4$-dimensional $p$-adic analytic symplectic manifold at a rank $0$ non-degenerate critical point is of the form $(g_1,g_2)$, where $g_1$ and $g_2$ can take one of the following three possible forms:
	\begin{enumerate}
		\item $g_1(x,\xi,y,\eta)=x^2+c_1\xi^2,g_2(x,\xi,y,\eta)=y^2+c_2\eta^2$, for some $c_1,c_2\in X_p$;
		\item $g_1(x,\xi,y,\eta)=x\eta+cy\xi,g_2(x,\xi,y,\eta)=x\xi+y\eta$, for some $c\in Y_p$;
		\item
		\[g_k(x,\xi,y,\eta)=\sum_{i=0}^{2}\mathcal{C}_i^k(c,t_1,t_2,a,b)x^iy^{2-i}+\sum_{i=0}^{2}\mathcal{D}_i^k(c,t_1,t_2,a,b)\xi^i\eta^{2-i},\]
		for $k\in\{1,2\}$, where $c$, $t_1$, $t_2$, $a$ and $b$ can take a fixed set of possible values (see \cite[Table 1]{CrePel-williamson} for the full specification).
	\end{enumerate}
	Here $(x,\xi,y,\eta)$ are $p$-adic local symplectic coordinates centered at the critical point of the $p$-adic analytic integrable system.
	We say that each of these local normal forms at a rank $0$ non-degenerate critical point are of \emph{class (1), (2) or (3)} respectively, depending on which of the three cases above they belong to. We denote each one of these local normal forms by its parameters: those of class (1) are denoted by $(c_1,c_2)$, those of class (2) by $c$, and those of class (3) by $(c,t_1,t_2,a,b)$.
\end{definition}

In our paper \cite{CrePel-williamson} we spoke of linear symplectic coordinates, but in view of our later paper \cite{CrePel-Darboux} our statements therein are also valid in a neighborhood of the point. Hence why in Theorem \ref{thm:new-main} and Definition \ref{def:classes} we speak of local symplectic coordinates and not of linear symplectic coordinates. We make this precise in Lemma \ref{lemma:darboux}.

\begin{definition}[Class R3 and I3$(c)$ normal forms]\label{ref:classes-deg}
	By \cite[Theorem 10.6]{CrePel-williamson}, the local normal form of a $p$-adic analytic integrable system on a $4$-dimensional $p$-adic analytic symplectic manifold at a rank $0$ degenerate critical point is of the form $(g_1,g_2)$, where $g_1$ and $g_2$ can take several possible forms. The ones which are interesting to us for the present paper are the class R3, given by $(x\xi+y\eta,y\xi)$, and the class I3$(c)$, where $c\in Y_p$, given by
	\[\left(x\eta+cy\xi+\frac{(1+c)y^2}{2},\frac{x^2-c\xi^2}{2}\right).\]
\end{definition}

\begin{definition}[$p$-adic local symplectomorphism {\cite[Definition 9.1]{CrePel-williamson}}]
	\letpprime. Let $(M_1,\omega_1)$ and $(M_2,\omega_2)$ be $p$-adic analytic symplectic manifolds. Let $m\in M_1$. A \emph{$p$-adic local symplectomorphism} $\phi:U_1\to U_2$ \emph{centered at $m$} is a $p$-adic analytic symplectomorphism between some open sets $U_1\subset M_1$ and $U_2\subset M_2$, such that $m\in U_1$.
\end{definition}

Now we are ready to state our main results about the rank $0$ singularities of coupled angular momenta.

\begin{maintheorem}[Number of normal forms at rank $0$ singularities]\label{thm:number}
	\letpprime. Let $\F$ be the parameter set in \eqref{eq:F} and let $(t,R_1,R_2)\in\F$. The $p$-adic coupled angular momentum $F_{t,R_1,R_2}=(J_{t,R_1,R_2},H_{t,R_1,R_2}):\sphere\times\sphere\to(\Qp)^2$ given by \eqref{eq:angular}
	is a $p$-adic analytic integrable system with exactly four rank $0$ critical points:
	\begin{equation}\label{eq:PQST}
		\left\{\begin{aligned}
			P & =(0,0,1,0,0,1); \\
			Q & =(0,0,-1,0,0,1); \\
			S & =(0,0,1,0,0,-1); \\
			T & =(0,0,-1,0,0,-1),
		\end{aligned}\right.
	\end{equation}
	which are non-degenerate for almost all values of the parameters. Furthermore:
	\begin{enumerate-roman}
		\item \textup{(Number of non-equivalent local normal forms, $p\not\equiv 3\mod 4$ at $P$ or $Q$)} If $p\not\equiv 3\mod 4$, the local normal form of $F_{t,R_1,R_2}$ at $P$ is non-degenerate and unique up to $p$-adic local symplectomorphisms centered at the critical point, independently of the parameters $t$, $R_1$ and $R_2$. The same statement holds when replacing $P$ by $Q$.
		\item \textup{(Number of non-equivalent local normal forms, $p\equiv 3\mod 4$ at $P$ or $Q$)} If $p\equiv 3\mod 4$, there are exactly three non-degenerate local normal forms of $F_{t,R_1,R_2}$ at $P$ up to $p$-adic local symplectomorphisms centered at the critical point. The same statement holds when replacing $P$ by $Q$.
		\item \textup{(Number of non-equivalent local normal forms, $p=2$ at $S$ or $T$)} If $p=2$, the local normal form of $F_{t,R_1,R_2}$ at $S$ is non-degenerate and unique up to $p$-adic local symplectomorphisms centered at the critical point, independently of the parameters $t$, $R_1$ and $R_2$. The same statement holds when replacing $S$ by $T$.
		\item \textup{(Number of non-equivalent local normal forms, $p\equiv 1\mod 4$ at $S$ or $T$)} If $p\equiv 1\mod 4$, there are exactly $4$ non-degenerate local normal forms and one degenerate local normal form of $F_{t,R_1,R_2}$ at $S$ up to $p$-adic local symplectomorphisms centered at the critical point. The same statement holds when replacing $S$ by $T$.
		\item \textup{(Number of non-equivalent local normal forms, $p\equiv 3\mod 4$ at $S$ or $T$)} If $p\equiv 3\mod 4$, there are exactly $8$ non-degenerate local normal forms and one degenerate local normal form of $F_{t,R_1,R_2}$ at $S$ up to $p$-adic local symplectomorphisms centered at the critical point. The same statement holds when replacing $S$ by $T$.
	\end{enumerate-roman}
\end{maintheorem}

\begin{maintheorem}[Description of linear normal forms at rank $0$ singularities]\label{thm:main}
	\letpprime. Let $\F$ be the parameter set in \eqref{eq:F} and let $(t,R_1,R_2)\in\F$. Let $F_{t,R_1,R_2}=(J_{t,R_1,R_2},H_{t,R_1,R_2})\in\sphere\times\sphere\to(\Qp)^2$ be the $p$-adic coupled angular momentum given by \eqref{eq:angular}. Let $P,Q,S,T$ be the points given in \eqref{eq:PQST}. Then the following statements hold.
	\begin{enumerate-roman}
		\item The points $P$ and $Q$ are in the type $(1,1)$ of class (1) (i.e. the local normal form is $(x^2+\xi^2,y^2+\eta^2)$) if $p\equiv 1\mod 4$ or $p=2$, and in one of the types
		\[(1,1),\;(1,p^2),\;(p^2,p^2)\]
		of class (1) if $p\equiv 3\mod 4$.
		\item Let $m$ be either $S$ or $T$. Then:
		\begin{enumerate}
			\item If $p=2$, $m$ is always in the type $(1,1)$ of class (1).
			\item If $p\equiv 1\mod 4$ and $m$ is non-degenerate, it is in the type $(1,1)$ of class (1) or any type of class (2); if it is degenerate, it is in the type \emph{R3}.
			\item If $p\equiv 3\mod 4$ and $m$ is non-degenerate, it is in one of the types $(1,1)$, $(1,p^2)$ or $(p^2,p^2)$ of class (1), the type $1$ of class (2), or the types
			\[(p,-1,0,1,0),\;(p,-1,0,1,1),\;(-p,-1,0,1,0),\;(-p,-1,0,1,1)\]
			of class (3); if it is degenerate, it is in the type \emph{I3$(-1)$}.
		\end{enumerate}
	\end{enumerate-roman}
\end{maintheorem}

\subsection{Classification of rank $1$ singularities}

In \cite[Theorem A]{CrePel-williamson}, we also classify rank $1$ critical points of integrable systems. All local normal forms of these points, in the non-degenerate case, are $(x^2+c'\xi^2+\ocal(3),\eta+\ocal(2))$, where $c'\in X_p$. The following two are our main results concerning the rank $1$ singularities of the coupled angular momentum.

\begin{maintheorem}[Number of normal forms at rank $1$ singularities]\label{thm:number1}
	\letpprime. Let $\F$ be the parameter set in \eqref{eq:F} and let $(t,R_1,R_2)\in\F$. The $p$-adic coupled angular momentum $F_{t,R_1,R_2}=(J_{t,R_1,R_2},H_{t,R_1,R_2}):\sphere\times\sphere\to(\Qp)^2$ given by \eqref{eq:angular}
	has critical points of rank $1$ at the following positions:
	\begin{enumerate-alph}
		\item if $t=0$, the points of the form $(0,0,\pm 1,x_2,y_2,z_2)$ and $(x_1,y_1,z_1,0,0,\pm 1)$, which can have two different local normal forms;
		\item if $t\ne 0$, the points of the form $(x_1,y_1,z_1,x_2,y_2,z_2)$, where
		\[
		\left\{
		\begin{aligned}
			z_1 & =\frac{1}{2}\left(\frac{t}{ck}+\frac{ck}{t}-\frac{ckt}{(1-t-c)^2}\right); \\
			z_2 & =\frac{1}{2}\left(\frac{t(1-t-c)}{c^2k^2}-\frac{t}{1-t-c}-\frac{1-t-c}{t}\right),
		\end{aligned}
		\right.
		\]
		$x_1$ and $y_1$ are such that $x_1^2+y_1^2+z_1^2=1$,
		\[
		\left\{
		\begin{aligned}
			x_2 & =\frac{1-t-c}{ck}x_1; \\
			y_2 & =\frac{1-t-c}{ck}y_1,
		\end{aligned}
		\right.
		\]
		$c\in\Qp$ with $c\ne 0$, and $k=R_2/R_1$, which can have all the possible local normal forms ($7$ if $p\equiv 1\mod 4$, $5$ if $p\equiv 3\mod 4$, and $11$ if $p=2$);
		\item if $t=1$, the points of the form $(x_1,y_1,z_1,x_1,y_1,z_1)$ and $(x_1,y_1,z_1,-x_1,-y_1,-z_1)$, which can have two different local normal forms.
	\end{enumerate-alph}
\end{maintheorem}

\begin{maintheorem}[Description of normal forms at rank $1$ singularities]\label{thm:main1}
	\letpprime. Let $\F$ be the parameter set in \eqref{eq:F} and let $(t,R_1,R_2)\in\F$. Let $F_{t,R_1,R_2}=(J_{t,R_1,R_2},H_{t,R_1,R_2}):\sphere\times\sphere\to(\Qp)^2$ be the $p$-adic coupled angular momentum given by \eqref{eq:angular}. Then the local normal form of the rank $1$ critical points described in Theorem \ref{thm:number1} is $(x^2+c'\xi^2+\ocal(3),\eta+\ocal(2))$, where $c'\in\{1,p^2\}$ for the points in parts (a) and (c), and $c'\in X_p$ arbitrary for the points in part (b).
\end{maintheorem}

\subsection{Works on $p$-adic geometry and symplectic geometry}\label{sec:works}

The field of $p$-adic geometry is extensive: see Lurie's lecture \cite{Lurie}, the book by Scholze-Weinstein \cite{SchWei} and the references therein. We recommend \cite{BreFre,CLH,DKKV,DKKVZ,GKPSW} for works about the role of $p$-adic numbers in mathematical physics; these numbers have been used in quantum mechanics \cite{Gosson,DjoDra,Dragovich-quantum,Dragovich-harmonic,RTVW,VlaVol}, string dynamics \cite{BFOW,FreOls,FreWit,FGZ,GarLop,Volovich} and cosmology \cite{HSSS,MarTed,Susskind}. For a construction of $p$-adic symplectic vector spaces and some of their invariants, see Zelenov \cite{Zelenov} and Hu-Hu \cite{HuHu}. We recommend the articles \cite{AloHoh,Eliashberg,Palmer,PelVuN-symplectic,Pelayo-hamiltonian,Pelayo-symplectic,Schlenk,Weinstein-symplectic} and the books \cite{HofZeh,MarRat,McDSal,OrtRat} for surveys on symplectic geometry and topology and integrable systems, and its connection to mechanics. See also Remark \ref{rem:works}.

\subsection*{Structure of the paper}

Section \ref{sec:general} proves some general results about the coupled angular momentum system. Section \ref{sec:PQ} computes the normal forms at the critical points $P$ and $Q$. Sections \ref{sec:ST-outer}, \ref{sec:ST-inner} and \ref{sec:ST-limit} compute the normal forms at the critical points $S$ and $T$, for different values of the parameters. Section \ref{sec:01} does the same computation for $t\in\{0,1\}$, which require a different proof than the rest of values of $t$. Section \ref{sec:rank1} computes the normal forms at the critical points of rank $1$. Section \ref{sec:proofs} contains the proofs of the theorems in the introduction. In Section \ref{sec:final} we make some final remarks. We close the paper with two appendices (Sections \ref{sec:appendix} and \ref{sec:real}) which recall the basics of $p$-adic numbers and give a quick review of the real coupled angular momentum system.

\bigskip
\noindent \textbf{Acknowledgments.} The first author is grateful to the School of Mathematics of the Complutense University for the hospitality during a visit in May 2025 where part of this paper was written. The first author is funded by grant PID2022-137283NB-C21 of MCIN/AEI/ 10.13039/501100011033 / FEDER, UE. The second author is funded by a FBBVA (Bank Bilbao Vizcaya Argentaria Foundation) Grant for Scientific Research Projects with title \textit{From Integrability to Randomness in Symplectic and Quantum Geometry}. The second author thanks Luis Crespo and Francisco Santos from the University of Cantabria for the hospitality during July and August of 2025 when part of this paper was written, the Dean of the School of Mathematical Sciences Antonio Br\'u, and the Chair of the Department of Algebra, Geometry and Topology at the Complutense University of Madrid, Rutwig Campoamor, for their support and excellent resources he is being provided with to carry out the FBBVA project.

\section{Preliminary results}\label{sec:general}

In this section we prove some preliminary results about the $p$-adic coupled angular momentum system in Definition \ref{def:angular} which will hold for all prime numbers. In our recent paper \cite[Theorem B]{CrePel-Darboux}, we proved that it is always possible to choose local symplectic coordinates around a point in a $p$-adic analytic symplectic manifold, that is, we prove a $p$-adic analog of Darboux's Theorem in real symplectic geometry and use it to derive a strong global classification \cite[Theorem D]{CrePel-Darboux} which is a symplectic version of a result of Serre \cite[Th\'eor\`eme 1]{Serre}. This means that all the normal forms in the present paper can be extended to have the standard form for the symplectic form at a whole neighborhood of the critical point, instead of only at the critical point.

\begin{lemma}\label{lemma:darboux}
	\letpprime. Let $(M,\omega)$ be a $4$-dimensional $p$-adic analytic symplectic manifold. Let $F=(f_1,f_2):M\to(\Qp)^2$ be a $p$-adic analytic integrable system and let $m$ be a critical point of $F$. Let $(g_1,g_2)$ be the linear normal form of $F$ at $m$ according to \cite[Theorem A]{CrePel-williamson}. Then $(g_1,g_2)$ is also the local normal form of $F$ at $m$, that is, there exist local symplectic coordinates around $m$ such that $F-F(m)=B\circ(g_1,g_2)+\ocal(3)$ for some matrix $B\in\M_2(\Qp)$.
\end{lemma}

\begin{proof}
	By \cite[Theorem B]{CrePel-Darboux}, there exist local symplectic coordinates $(x,\xi,y,\eta)$ defined in an open subset $U\subset M$ which contains $m$. We apply \cite[Theorem A]{CrePel-williamson} to $F$ written in the new coordinates $(x,\xi,y,\eta)$. This gives us linear symplectic coordinates $(x',\xi',y',\eta')$ in which $F$ attains its normal form. Since the linear normal form of a system is unique, that normal form must be $(g_1,g_2)$ itself.
	
	The coordinates $(x',\xi',y',\eta')$ are obtained making a linear change of coordinates in $(x,\xi,y,\eta)$; let $S$ be the matrix of this coordinate change. Since both sets of coordinates are linearly symplectic, $S$ is a symplectic matrix. This means that the coordinates $(x',\xi',y',\eta')$ can be obtained multiplying $(x,\xi,y,\eta)$ by the symplectic matrix $S$, and the latter are symplectic on all of $U$, hence the former must also be symplectic on all of $U$ and $(g_1,g_2)$ is a local normal form of $F$, as we wanted.
\end{proof}

In Definition \ref{def:angular} it is implicitly assumed that the $p$-adic coupled angular momentum is a $p$-adic analytic integrable system, which requires proving:

\begin{itemize}
	\item that the Poisson bracket of $J_{t,R_1,R_2}$ and $H_{t,R_1,R_2}$ vanishes everywhere,
	\item that $\dd J_{t,R_1,R_2}$ and $\dd H_{t,R_1,R_2}$ are linearly independent on a dense subset of $\sphere\times\sphere$ (the definition of $p$-adic analytic integrable system is given by Pelayo-Voevodsky-Warren \cite[Definition 7.1]{PVW}, and we use a slight refinement of it which is less restrictive \cite[Definition 3.3]{CrePel-JC}).
\end{itemize}

Given a $p$-adic analytic function on a $p$-adic analytic symplectic manifold $f:(M,\omega)\to\Qp$, the \emph{Hamiltonian vector field} $X_f$ associated to it is defined by the same expressions as in the real case. Similarly for the Poisson brackets $\{\cdot,\cdot\}$.

\begin{lemma}\label{lemma:poisson-sphere}
	\letpprime. Let $\sphere$ be the $p$-adic sphere given by \eqref{eq:sphere}. Let $\omega_1,\omega_2$ be respectively the standard $p$-adic analytic symplectic forms on the first and second factor of $\sphere\times\sphere$. If $(x_1,y_1,z_1,x_2,y_2,z_2)$ are the coordinates on $(\sphere\times\sphere,R_1\omega_1\oplus R_2\omega_2)$, then
	$\{x_i,y_i\}=\frac{z_i}{R_i};\{y_i,z_i\}=\frac{x_i}{R_i};\{z_i,x_i\}=\frac{y_i}{R_i},$ for every $i\in\{1,2\}$.
\end{lemma}

\begin{proof}
	For any $p$-adic analytic vector field on a $p$-adic analytic symplectic manifold $(M,\omega)$, we denote by $\imath(X)\omega$ the contraction $\omega(X,\cdot)$ of the symplectic form by $X$. From the expression of $\omega_i$, which follows from Definition \ref{def:angular}, we obtain $\omega_i=\frac{1}{z_i}\dd x_i\wedge\dd y_i$, and $\imath\left(z_i\frac{\partial}{\partial x_i}\right)\omega_i=\dd y_i.$
	Hence
	$\imath\left(\frac{z_i}{R_i}\frac{\partial}{\partial x_i}\right)(R_1\omega_1\oplus R_2\omega_2)=\dd y_i$
	and it follows that
	$X_{y_i}=\frac{z_i}{R_i}\frac{\partial}{\partial x_i}.$
	Finally:
	$\{x_i,y_i\}=\dd x_i(X_{y_i})=\dd x_i\left(\frac{z_i}{R_i}\frac{\partial}{\partial x_i}\right)=\frac{z_i}{R_i}.$
	The other two identities in the statement follow from the fact that permuting cyclically the coordinates of $\sphere$, that is, changing $(x_i,y_i,z_i)$ to $(y_i,z_i,x_i)$, does not change the symplectic form.
\end{proof}

\begin{figure}
	\includegraphics[width=0.45\linewidth]{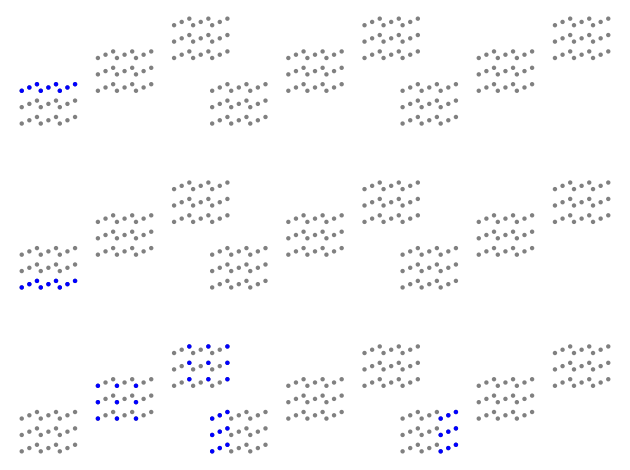}\raisebox{0.15\linewidth}{$\times$}\includegraphics[width=0.45\linewidth]{3adic-sphere}
	
	\bigskip
	\includegraphics[width=0.45\linewidth]{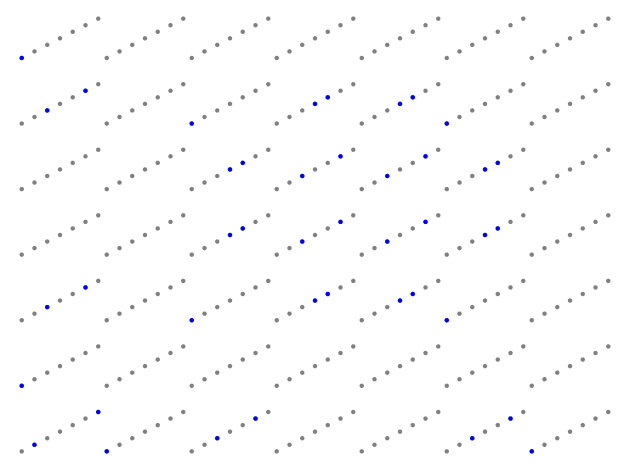}\raisebox{0.15\linewidth}{$\times$}\includegraphics[width=0.45\linewidth]{7adic-sphere}
	\caption{Symbolic representation of $\sphere\times\sphere$, the phase space of the $p$-adic coupled angular momentum, for $p=3$ and $p=7$ (see Definition \ref{def:angular} for the definition of $\sphere$). The grey points represent $(\Zp)^3$, and the blue points are those being in $\sphere$. Each point has radius $1/9$ above and $1/7$ below.}
	\label{fig:padic-spheres}
\end{figure}

\begin{proposition}\label{prop:poisson-cam}
	\letpprime. Let $\F$ be the parameter set in \eqref{eq:F} and let $(t,R_1,R_2)\in \F$. Let $F_{t,R_1,R_2}=(J_{t,R_1,R_2},H_{t,R_1,R_2}):\sphere\times\sphere\to(\Qp)^2$ be the $p$-adic coupled angular momentum system given by \eqref{eq:angular}. Then $\{J_{t,R_1,R_2},H_{t,R_1,R_2}\}=0$.
\end{proposition}

\begin{proof}
	The proof is the same as in the real case. Since the symplectic form does not mix the coordinates of the two $p$-adic spheres, $z_1$ commutes with itself and with $z_2$, hence also with $J_{t,R_1,R_2}$. The same happens with the term $z_1z_2$ in $H_{t,R_1,R_2}$. It is left to show that $x_1x_2+y_1y_2$ commutes with $J_{t,R_1,R_2}$:
	\begin{align*}
		\Big\{x_1x_2+y_1y_2,J_{t,R_1,R_2}(x_1,y_1,z_1,x_2,y_2,z_2)\Big\} & =\Big\{x_1x_2+y_1y_2,R_1z_1+R_2z_2\Big\} \\
		& =R_1(\{x_1,z_1\}x_2+\{y_1,z_1\}y_2) \\
		& \phantom{=}+R_2(x_1\{x_2,z_2\}+y_1\{y_2,z_2\}) \\
		& =R_1\left(-\frac{y_1x_2}{R_1}+\frac{x_1y_2}{R_1}\right)+R_2\left(-\frac{x_1y_2}{R_2}+\frac{y_1x_2}{R_2}\right) \\
		& =0,
	\end{align*}
	where in the second equality we have used the formula for the Poisson bracket of a product
	$\{fg,h\}=\{f,h\}g+f\{g,h\}$
	and in the third we are using Lemma \ref{lemma:poisson-sphere}.
\end{proof}

\begin{proposition}\label{prop:rank0}
	\letpprime. The $p$-adic coupled angular momentum system $F_{t,R_1,R_2}=(J_{t,R_1,R_2},H_{t,R_1,R_2}):\sphere\times\sphere\to(\Qp)^2$ given by \eqref{eq:angular} has exactly four rank $0$ critical points: the points $P$, $Q$, $S$ and $T$ defined in \eqref{eq:PQST}.
\end{proposition}

\begin{proof}
	A critical point of rank $0$ must have $\dd J_{t,R_1,R_2}=0$, that is, $R_1\dd z_1+R_2\dd z_2=0,$ which implies $\dd z_1=\dd z_2=0.$ On the $p$-adic sphere $\sphere$ we have $x_1\dd x_1+y_1\dd y_1+z_1\dd z_1=0,$ so $\dd z_1=0$ implies $x_1=y_1=0$. The only points in the $p$-adic sphere with $x_1=y_1=0$ are the poles $(0,0,1)$ and $(0,0,-1)$. Analogously, the point $(x_2,y_2,z_2)$ must be $(0,0,1)$ or $(0,0,-1)$, and we are done.
\end{proof}

Now we compute the Hessians of $J_{t,R_1,R_2}$ and $H_{t,R_1,R_2}$ at the four rank $0$ critical points.

\begin{proposition}\label{prop:hessians}
	\letpprime. Let $\F$ be the parameter set in \eqref{eq:F} and let $(t,R_1,R_2)\in \F$. Let $F_{t,R_1,R_2}=(J_{t,R_1,R_2},H_{t,R_1,R_2}):\sphere\times\sphere\to(\Qp)^2$ be the $p$-adic coupled angular momentum system given by \eqref{eq:angular}. The Hessian of $J_{t,R_1,R_2}:\sphere\times\sphere\to\Qp$ at the critical point $(0,0,z_1,0,0,z_2)$, where $z_1,z_2\in\{1,-1\}$, in terms of the coordinates $(x_1,y_1,x_2,y_2)$, is
	\begin{equation}\label{eq:MJ}
		M_{J_{t,R_1,R_2}}(z_1,z_2)=\begin{pmatrix}
			-\frac{R_1}{z_1} & 0 & 0 & 0 \\
			0 & -\frac{R_1}{z_1} & 0 & 0 \\
			0 & 0 & -\frac{R_2}{z_2} & 0 \\
			0 & 0 & 0 & -\frac{R_2}{z_2}
		\end{pmatrix}
	\end{equation}
	The Hessian of $H_{t,R_1,R_2}:\sphere\times\sphere\to\Qp$ at the critical point $(0,0,z_1,0,0,z_2)$, where $z_1,z_2\in\{1,-1\}$, in terms of the coordinates $(x_1,y_1,x_2,y_2)$, is
	\begin{equation}\label{eq:MH}
		M_{H_{t,R_1,R_2}}(z_1,z_2)=\begin{pmatrix}
			\frac{t-tz_2-1}{z_1} & 0 & t & 0 \\
			0 & \frac{t-tz_2-1}{z_1} & 0 & t \\
			t & 0 & -\frac{tz_1}{z_2} & 0 \\
			0 & t & 0 & -\frac{tz_1}{z_2}
		\end{pmatrix}.
	\end{equation}
\end{proposition}

\begin{proof}
	Since $(x_1)^2+(y_1)^2+(z_1)^2=1$, we have that
	$\dd z_1=-\frac{x_1}{z_1}\dd x_1-\frac{y_1}{z_1}\dd y_1,$
	the Hessian of $z_1$ in the coordinates $(x_1,y_1)$ is:
	\[\begin{pmatrix}
		-\frac{1}{z_1}-\frac{(x_1)^2}{(z_1)^3} & -\frac{x_1y_1}{(z_1)^3} \\
		-\frac{x_1y_1}{(z_1)^3} & -\frac{1}{z_1}-\frac{(y_1)^2}{(z_1)^3}
	\end{pmatrix}\]
	which after substituting $x_1=y_1=0$ gives
	\[\begin{pmatrix}
		-\frac{1}{z_1} & 0 \\
		0 & -\frac{1}{z_1}
	\end{pmatrix}.\]
	Now we have
	\begin{align*}
		\dd^2 J_{t,R_1,R_2}(x_1,y_1,z_1,x_2,y_2,z_2) & =\dd^2(R_1z_1+R_2z_2) \\
		& =\begin{pmatrix}
			-\frac{R_1}{z_1} & 0 & 0 & 0 \\
			0 & -\frac{R_1}{z_1} & 0 & 0 \\
			0 & 0 & -\frac{R_2}{z_2} & 0 \\
			0 & 0 & 0 & -\frac{R_2}{z_2}
		\end{pmatrix}
	\end{align*}
	and
	\begin{align*}
		\dd^2 H_{t,R_1,R_2}(x_1,y_1,z_1,x_2,y_2,z_2) & =(1-t)\dd^2 z_1+t\dd^2(x_1x_2+y_1y_2)+t\dd^2(z_1z_2) \\
		& =(1-t)\begin{pmatrix}
			-\frac{1}{z_1} & 0 & 0 & 0 \\
			0 & -\frac{1}{z_1} & 0 & 0 \\
			0 & 0 & 0 & 0 \\
			0 & 0 & 0 & 0
		\end{pmatrix}
		+t\begin{pmatrix}
			0 & 0 & 1 & 0 \\
			0 & 0 & 0 & 1 \\
			1 & 0 & 0 & 0 \\
			0 & 1 & 0 & 0
		\end{pmatrix} \\
		& \phantom{=}
		+t\begin{pmatrix}
			-\frac{z_2}{z_1} & 0 & 0 & 0 \\
			0 & -\frac{z_2}{z_1} & 0 & 0 \\
			0 & 0 & -\frac{z_1}{z_2} & 0 \\
			0 & 0 & 0 & -\frac{z_1}{z_2}
		\end{pmatrix} \\
		& =\begin{pmatrix}
			\frac{t-tz_2-1}{z_1} & 0 & t & 0 \\
			0 & \frac{t-tz_2-1}{z_1} & 0 & t \\
			t & 0 & -\frac{tz_1}{z_2} & 0 \\
			0 & t & 0 & -\frac{tz_1}{z_2}
		\end{pmatrix}.\qedhere
	\end{align*}
\end{proof}

Let
\[k=\frac{R_2}{R_1}.\]
By the condition on $R_1$ and $R_2$ in Definition \ref{def:angular}, we have that \[\ord_p(k)<0.\] The effect of changing $R_2$ in the coupled angular momentum system while leaving $k$ constant is to multiply $J_{t,R_1,R_2}$ and $\omega$ by a factor of $R_2$.

\begin{lemma}\label{lemma:eigenvectors}
	\letpprime. Let $\F$ be the parameter set in \eqref{eq:F}. Let $(t,R_1,R_2)\in\F$ and let $k=R_2/R_1$. Let $z_1,z_2\in\{1,-1\}$. Let $\lambda$ and $\mu$ be the roots of
	\[x^2+\ii(k(tz_2-t+1)+tz_1)x+kt(t-1)z_1.\]
	Let $M_{H_{t,R_1,R_2}}(z_1,z_2)$ be the matrix in \eqref{eq:MH}.
	Let $\Omega_{R_1,R_2}(z_1,z_2)$ be the matrix of $\omega=R_1\omega_1\oplus R_2\omega_2$ at $(0,0,z_1,0,0,z_2)$:
	\begin{equation}\label{eq:omega}
		\Omega_{R_1,R_2}(z_1,z_2)=\begin{pmatrix}
			0 & \frac{R_1}{z_1} & 0 & 0 \\
			-\frac{R_1}{z_1} & 0 & 0 & 0 \\
			0 & 0 & 0 & \frac{R_2}{z_2} \\
			0 & 0 & -\frac{R_2}{z_2} & 0
		\end{pmatrix}
	\end{equation}
	Let
	\[A=R_2\Omega_{R_1,R_2}(z_1,z_2)^{-1}M_{H_{t,R_1,R_2}}(z_1,z_2)=\begin{pmatrix}
		0 & -k(t-tz_2-1) & 0 & -ktz_1 \\
		k(t-tz_2-1) & 0 & ktz_1 & 0 \\
		0 & -tz_2 & 0 & tz_1 \\
		tz_2 & 0 & -tz_1 & 0
	\end{pmatrix}.\]
	Then the matrix $A$ has four eigenvalues $\lambda,-\lambda,\mu,-\mu$ and:
	\begin{itemize}
		\item the eigenvector of $A$ corresponding to $\lambda$ is $(tz_1-\ii\lambda,-\ii tz_1-\lambda,tz_2,-\ii tz_2)$;
		\item the eigenvector of $A$ corresponding to $-\lambda$ is $(tz_1-\ii\lambda,\ii tz_1+\lambda,tz_2,\ii tz_2)$;
		\item the eigenvector of $A$ corresponding to $\mu$ is $(tz_1-\ii\mu,-\ii tz_1-\mu,tz_2,-\ii tz_2)$;
		\item the eigenvector of $A$ corresponding to $-\mu$ is $(tz_1-\ii\mu,\ii tz_1+\mu,tz_2,\ii tz_2)$.
	\end{itemize}
\end{lemma}

\begin{proof}
	We can check by computation that the characteristic polynomial of $A$ is
	\[(x^2+\ii(k(tz_2-t+1)+tz_1)x+kt(t-1)z_1)(x^2-\ii(k(tz_2-t+1)+tz_1)x+kt(t-1)z_1).\]
	The roots of the first factor are $\lambda$ and $\mu$, and those of the second factor are their opposites. This proves the first part. For the eigenvectors, if we multiply $A$ times the first vector, we get
	\begin{align*}
		& \phantom{=}\;\; (-k(t-tz_2-1)(-\ii tz_1-\lambda)+\ii kt^2z_1z_2,k(t-tz_2-1)(tz_1-\ii\lambda)+kt^2z_1z_2,\lambda tz_2,-\ii\lambda tz_2) \\
		& =(k(t-tz_2-1)\lambda+\ii k(t-tz_2-1)tz_1+\ii kt^2z_1z_2, \\
		& \phantom{=}\;\; -\ii k(t-tz_2-1)\lambda+ k(t-tz_2-1)tz_1+kt^2z_1z_2,\lambda tz_2,-\ii\lambda tz_2) \\
		& =(-k(tz_2-t+1)\lambda+\ii k(t-1)tz_1,\ii k(tz_2-t+1)\lambda+ k(t-1)tz_1,\lambda tz_2,-\ii\lambda tz_2) \\
		& =(-\ii\lambda^2+tz_1\lambda,-\lambda^2-\ii tz_1\lambda,\lambda tz_2,-\ii\lambda tz_2) \\
		& =\lambda(tz_1-\ii\lambda,-\ii tz_1-\lambda,tz_2,-\ii tz_2).
	\end{align*}
	The same happens with the rest of vectors, with the corresponding eigenvalues.
\end{proof}

\section{Normal forms of $F_{t,R_1,R_2}$ at $P$ and $Q$}\label{sec:PQ}

Now we compute the linear normal forms of the coupled angular momentum system $F_{t,R_1,R_2}:\sphere\times\sphere\to(\Qp)^2$ given by \eqref{eq:angular} at the rank $0$ critical points. We first do this at the points $P$ and $Q$, which have $z_2=1$.

\begin{lemma}\label{lemma:eigenvalues}
	\letpprime. Let $\F$ be the parameter set in \eqref{eq:F} and let $(t,R_1,R_2)\in \F$. Let $M_{J_{t,R_1,R_2}}(z_1,z_2)$ and $M_{H_{t,R_1,R_2}}(z_1,z_2)$ be the matrices in \eqref{eq:MJ} and \eqref{eq:MH} respectively. Let $\Omega_{R_1,R_2}(z_1,z_2)$ be the matrix of $\omega=R_1\omega_1\oplus R_2\omega_2$ at $(0,0,z_1,0,0,z_2)$ as in equation \eqref{eq:omega}.
	Let $k=R_2/R_1$ and $z_1\in\{1,-1\}$. The characteristic polynomial of
	\begin{equation}\label{eq:omegaM}
		R_2\Omega_{R_1,R_2}(z_1,1)^{-1}M_{H_{t,R_1,R_2}}(z_1,1)
	\end{equation}
	has four roots \[\Big\{\lambda,-\lambda,\mu,-\mu\Big\}\in\Qp[\ii],\] such that
	\[\ord_p(\lambda)=\ord_p(k)\]
	and
	\[\ord_p(\mu)=\ord_p(t(t-1)).\]
\end{lemma}

\begin{proof}
	By Lemma \ref{lemma:eigenvectors}, $\lambda$ and $\mu$ are the roots of
	\[P(x)=x^2+\ii(k+tz_1)x+kt(t-1)z_1.\]
	The discriminant of this equation is
	\[-(k+tz_1)^2-4kt(t-1)z_1,\]
	which has order $2\ord_p(k)$, because $\ord_p(k)<0$ and $\ord_p(t)\ge 0$. Hence it has a square root in $\Qp[\ii]$ of order $\ord_p(k)$ with the same leading digit as $-\ii k$. We call $\lambda$ the root of $P(x)$ given by this square root of $\Delta$, and $\mu$ the other root. Now we have that $\ord_p(\lambda)=\ord_p(k)$ and
	\[\ord_p(\mu)=\ord_p(kt(t-1)z_1)-\ord_p(k)=\ord_p(t(t-1)).\qedhere\]
\end{proof}

\begin{corollary}\label{cor:normalform}
	\letpprime. Let $\F$ be the parameter set in \eqref{eq:F} and let $(t,R_1,R_2)\in \F$. Let $F_{t,R_1,R_2}=(J_{t,R_1,R_2},H_{t,R_1,R_2}):\sphere\times\sphere\to\Qp^2$ be the $p$-adic coupled angular momentum system given by expression \eqref{eq:angular}. Let $P,Q$ be the points defined in \eqref{eq:PQST}, which by Proposition \ref{prop:rank0} are rank zero non-degenerate critical points of $F_{t,R_1,R_2}$. Let $m\in\{P,Q\}$. Let $X_p$ be the set given in Definition \ref{def:classes}. If $t$ is different from $0$ and $1$, there exist a matrix $B\in\M_2(\Qp)$, $c_1,c_2\in X_p$ and local coordinates $(x,\xi,y,\eta)$ centered at $m$ such that the $p$-adic symplectic form is given by $\dd x\wedge\dd\xi+\dd y\wedge\dd\eta$ and
	\begin{align*}
		\widetilde{F}_{t,R_1,R_2}(x,\xi,y,\eta) & =B\circ(F_{t,R_1,R_2}(x,\xi,y,\eta)-F_{t,R_1,R_2}(m)) \\
		& =\frac{1}{2}(x^2+c_1\xi^2,y^2+c_2\eta^2)+\ocal((x,\xi,y,\eta)^3).
	\end{align*}
\end{corollary}

\begin{proof}
	That there exist linear symplectic coordinates with this property follows from Lemma \ref{lemma:eigenvalues} and \cite[Theorem A]{CrePel-williamson}, because dividing the matrix \eqref{eq:omegaM} by the constant $R_2$, which is in $\Qp$, only divides the eigenvalues by the same constant and $\lambda^2,\mu^2\in\Qp$ corresponds to class (1) in the Williamson classification. The existence of local symplectic coordinates with the same property follows from Lemma \ref{lemma:darboux}.
\end{proof}

\begin{lemma}\label{lemma:eigenvectors-lambda}
	\letpprime. Let $\F$ be the parameter set in \eqref{eq:F} and let $(t,R_1,R_2)\in \F$. Let $F_{t,R_1,R_2}=(J_{t,R_1,R_2},H_{t,R_1,R_2}):\sphere\times\sphere\to\Qp^2$ be the $p$-adic coupled angular momentum system given by expression \eqref{eq:angular}. Let $P,Q$ be the points defined in \eqref{eq:PQST}, which by Proposition \ref{prop:rank0} are rank zero non-degenerate critical points of $F_{t,R_1,R_2}$. Then, if $t\notin\{0,1\}$, the value of $c_1$ in Corollary \ref{cor:normalform} is given as follows.
	\begin{enumerate}
		\item If $p\equiv 1\mod 4$ or $p=2$, then $c_1=1$, that is, the component is elliptic (i.e. of the form $(x^2+\xi^2)/2$).
		\item If $p\equiv 3\mod 4$, then $c_1=1$ if $\ord_p(R_1)$ is even and $c_1=p^2$ otherwise.
	\end{enumerate}
\end{lemma}

\begin{proof}
	If $p\equiv 1\mod 4$, then \[\Qp[\ii]=\Qp,\] which implies $\lambda\in\Qp$ and this component is elliptic (or hyperbolic, which is the same in the $p$-adic case).
	
	If $p=2$, a component with eigenvalues in $\Qp[\ii]$ must have a normal form $x^2+c_1\xi^2$ with a $c_1\in X_p$ such that $-c_1$ is the square of an element of $\Qp[\ii]$ outside $\Qp$. Among the elements of the set $X_p$, only $1$ satisfies the condition, hence $c_1=1$.
	
	Now we assume that $p\equiv 3\mod 4$. In this case, there are two elements $c_1\in X_p$ such that $-c_1$ is the square of an element of $\Qp[\ii]$: $1$ and $p^2$. In order to distinguish between them, we use the criterion in \cite[Proposition 4.6]{CrePel-williamson}. Let $A$ be the matrix in \eqref{eq:omegaM}. Lemma \ref{lemma:eigenvectors} gives us the eigenvector of $A$ corresponding to $\lambda$, which is
	\[(tz_1-\ii\lambda,-\ii tz_1-\lambda,t,-\ii t).\]
	We have that
	\begin{align*}
		v\tr\Omega \bar{v} & =(tz_1-\ii\lambda)R_1z_1(\ii tz_1+\lambda)-(-\ii tz_1-\lambda)R_1z_1(tz_1-\ii\lambda)+\ii R_2t^2+\ii R_2t^2 \\
		& =R_2\left(2\ii\frac{z_1}{k}(tz_1-\ii\lambda)^2+2\ii t^2\right).
	\end{align*}
	By Lemma \ref{lemma:eigenvalues}, $\ord_p(\lambda)=\ord_p(k)$ and $\ord_p(v\tr\Omega \bar{v})=\ord(R_2k)$.
	
	Now we use the formula in \cite[Proposition 4.6]{CrePel-williamson} with $a=b=\ii\lambda/R_2$. We need to check whether
	\[\frac{2a\lambda/R_2}{v\tr\Omega \bar{v}}=\frac{2\ii\lambda^2}{R_2^2v\tr\Omega \bar{v}}\]
	is in
	\begin{equation}\label{eq:eigenvectors-lambda}
		\DSq(\Qp,-\lambda^2)=\Big\{r^2-\lambda^2s^2:r,s\in\Qp\Big\},
	\end{equation}
	which is the same as $\DSq(\Qp,1)$. By \cite[Proposition 4.7]{CrePel-williamson}, it is in \eqref{eq:eigenvectors-lambda} if and only if $\ord_p(R_2k)$ is even, which is equivalent to saying that $\ord_p(R_1)$ is even. Hence, in this case we may take $c_1=1$, and if $\ord_p(R_1)$ is odd, we take $c_1=p^2$ instead.
\end{proof}

\begin{lemma}\label{lemma:eigenvectors-mu}
	\letpprime. Let $\F$ be the parameter set in \eqref{eq:F} and let $(t,R_1,R_2)\in \F$. Let $F_{t,R_1,R_2}=(J_{t,R_1,R_2},H_{t,R_1,R_2}):\sphere\times\sphere\to\Qp^2$ be the $p$-adic coupled angular momentum system given by expression \eqref{eq:angular}. Let $P,Q$ be the points defined in \eqref{eq:PQST}, which by Proposition \ref{prop:rank0} are rank zero non-degenerate critical points of $F_{t,R_1,R_2}$. Then, if $t\notin\{0,1\}$, the value of $c_2$ in Corollary \ref{cor:normalform} is given as follows.
	\begin{enumerate}
		\item If $p\equiv 1\mod 4$ or $p=2$, then $c_2=1$, that is, the component is elliptic (i.e. of the form $(y^2+\eta^2)/2$).
		\item If $p\equiv 3\mod 4$, then $c_2=1$ if $\ord_p(R_2)$ is even and $c_2=p^2$ otherwise.
	\end{enumerate}
\end{lemma}

\begin{proof}
	If $p\equiv 1\mod 4$, then \[\Qp[\ii]=\Qp\] and $\mu\in\Qp$, and the component is elliptic.
	
	If $p=2$, then we are in the same situation as with the other component: among the values in $X_p$ (recall that $X_p$ was defined in Definition \ref{def:classes}), only $c_2=1$ satisfies that $-c_2$ is the square of an element of $\Qp[\ii]$ outside $\Qp$.
	
	If $p\equiv 3\mod 4$, this condition is satisfied by $c_2=1$ and $c_2=p^2$, so we need the eigenvector criterion. Let $A$ be the $p$-adic matrix in \eqref{eq:omegaM}. Lemma \ref{lemma:eigenvectors} gives us the eigenvector of $A$ corresponding to $\lambda$, which is
	\[(tz_1-\ii\mu,-\ii tz_1-\mu,t,-\ii t).\]
	We have that
	\begin{align*}
		v\tr\Omega \bar{v} & =(tz_1-\ii\mu)R_1z_1(\ii tz_1+\mu)-(-\ii tz_1-\mu)R_1z_1(tz_1-\ii\mu)+\ii R_2t^2+\ii R_2t^2 \\
		& =R_2\left(2\ii\frac{z_1}{k}(tz_1-\ii\mu)^2+2\ii t^2\right).
	\end{align*}
	By Lemma \ref{lemma:eigenvalues}, $\ord_p(\mu)=\ord_p(t(t-1))$, hence the first term has order greater or equal than $2\ord_p(t)-\ord_p(k)>2\ord_p(t)$, and the order of $v\tr\Omega \bar{v}$ is $\ord_p(R_2)+2\ord_p(t)$, which has the same parity as $\ord_p(R_2)$.
	
	Now we use the formula in \cite[Proposition 4.6]{CrePel-williamson} with $a=b=\ii\mu/R_2$. The result is
	\[\frac{2a\mu/R_2}{v\tr\Omega \bar{v}}=\frac{2\ii\mu^2}{R_2^2v\tr\Omega \bar{v}}.\]
	This is in \[\DSq(\Qp,-\mu^2)=\DSq(\Qp,1)\] if and only if $R_2$ has even order, hence $c_2=1$ is valid in that case, and otherwise $c_2=p^2$.
\end{proof}

\section{Normal forms of $F_{t,R_1,R_2}$ at $S$ and $T$ in the outer region}\label{sec:ST-outer}

In Section \ref{sec:PQ} we computed the types of the critical points $P$ and $Q$ in \eqref{eq:PQST} of the $p$-adic coupled angular momentum $F_{t,R_1,R_2}$ given in \eqref{eq:angular}. Now we proceed to study the points $S$ and $T$ in \eqref{eq:PQST}. Unlike for $P$ and $Q$, the behavior of $F_{t,R_1,R_2}$ at $S$ and $T$ will vary greatly depending on the parameters $k$ and $t$. In order to classify this behavior, we divide the space of possible parameters in three regions: outer, inner and limit (see Figure \ref{fig:regions}). In this section we focus on the outer region.

The aforementioned regions play a similar role as the regions into which the interval $[0,1]$ is divided in the real case in \cite[Proposition 2.5]{LeFPel}, though in that case there are only two regions: in the inner region $S$ is a focus-focus point and $T$ is elliptic-elliptic, while in the outer region both are elliptic-elliptic.

\begin{figure}
	\includegraphics[trim=15cm 0 15cm 0,scale=0.8]{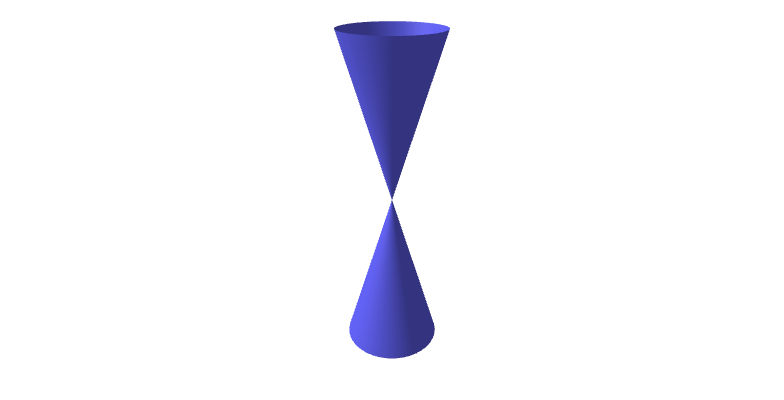}
	
	\includegraphics[scale=0.8]{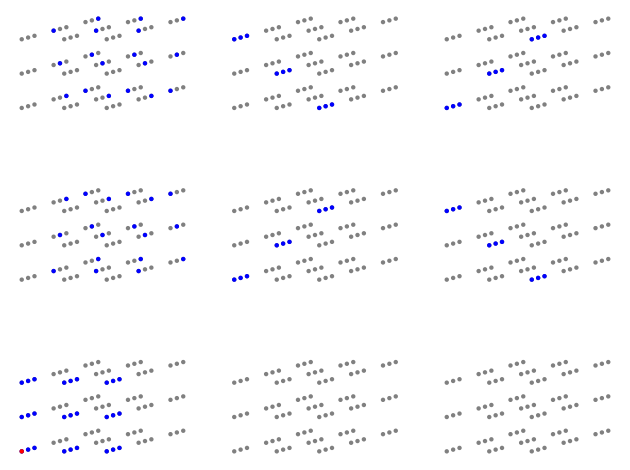}
	
	\caption{The fiber of a focus-focus critical value around the singularity in the real case (above) and the $p$-adic case (below, with the critical point in red). One of the rank $0$ critical points of the (real) coupled angular momentum system has this type for certain values of the parameter $t$; its $p$-adic equivalent $S=(1,0,0,-1,0,0)$ may or may not have the same type depending on the parameters.}
	\label{fig:pinched}
\end{figure}

\begin{definition}[Inner, outer and limit regions]\label{def:regions}
	\letpprime. We say that the pair of parameters $(k,t)\in \Qp\times\Zp$ \emph{is in the outer region} if $\ord_p(k(2t-1)^2)$ is negative, in the \emph{inner region} if this order is positive, and in the \emph{limit region} if this order is $0$.
\end{definition}

Note that, if $p=2$, all values of $(k,t)$ are in the outer region because $\ord_p(2t-1)$ is always zero. For $p=3$, a representation of the regions is in Figure \ref{fig:regions}.

\begin{figure}
	\includegraphics{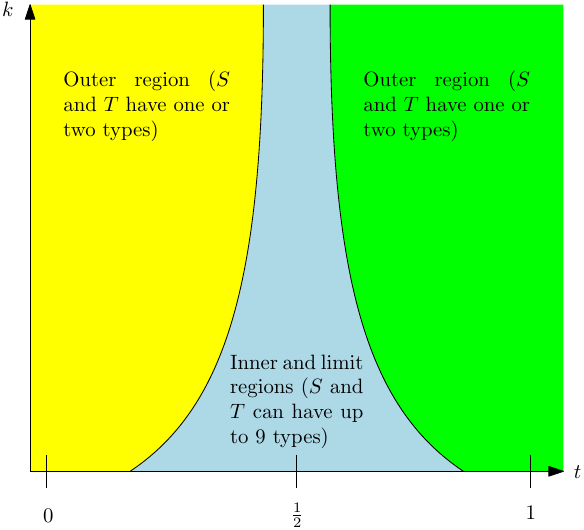}
	\includegraphics[scale=0.7]{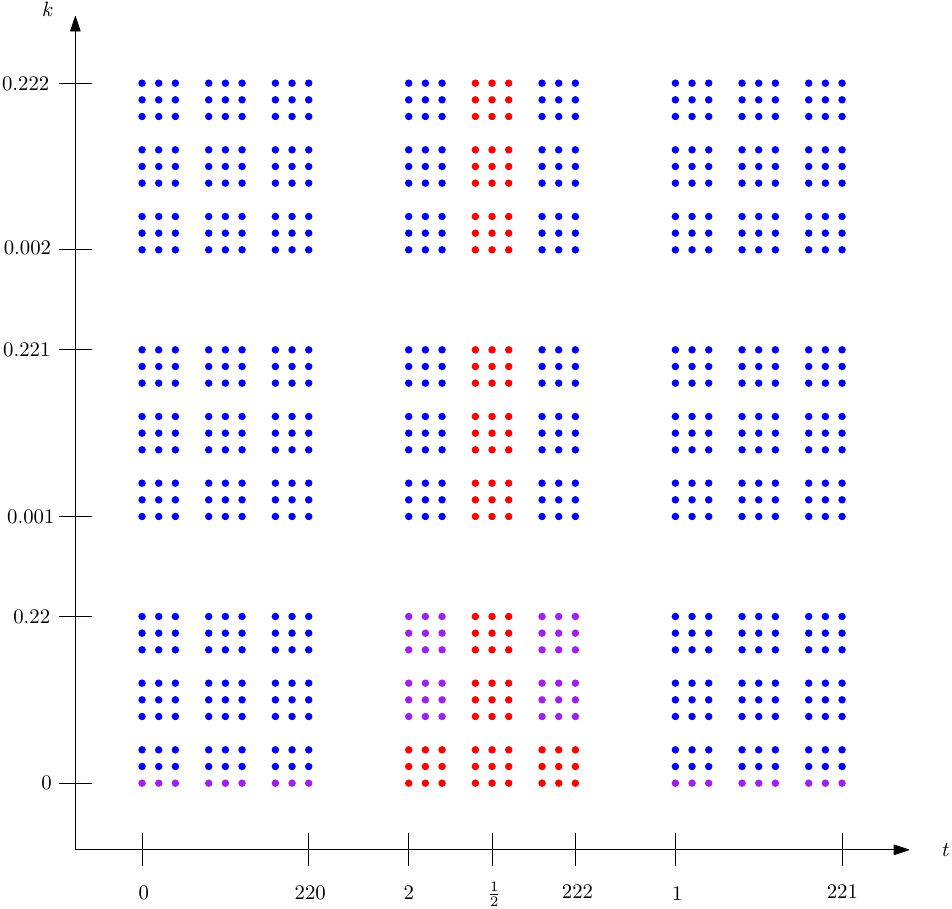}
	\caption{The outer (blue), inner (red) and limit (purple) regions of Definition \ref{def:regions} for $p=3$. The $x$-axis is the parameter $t$ where each point corresponds to a radius $1/27$, and the $y$-axis is the parameter $k$ where each point corresponds to a radius $1$. The bottom row of points is actually not a valid choice of parameters, because $k$ must not be in $\Zp$. The region in which the parameters are has an effect on the normal forms: in the inner and limit regions, according to Sections \ref{sec:ST-inner} and \ref{sec:ST-limit} respectively, $S$ and $T$ can have up to $5$ different forms, while in the outer region, according to Section \ref{sec:ST-outer}, they can only have one or two forms.}
	\label{fig:regions}
\end{figure}

We start with the outer region. As we will see, in that region the points $S$ and $T$ behave similarly to $P$ and $Q$. This is not surprising, because it is analogous to the situation in the real case.

\begin{lemma}\label{lemma:eigenvalues2}
	\letpprime. Let $\F$ be the parameter set in \eqref{eq:F} and let $(t,R_1,R_2)\in \F$. Let $M_{J_{t,R_1,R_2}}(z_1,z_2)$ and $M_{H_{t,R_1,R_2}}(z_1,z_2)$ be the two matrices in \eqref{eq:MJ} and \eqref{eq:MH} respectively. Let $\Omega_{R_1,R_2}(z_1,z_2)$ be the matrix of $\omega=R_1\omega_1\oplus R_2\omega_2$ as in equation \eqref{eq:omega}.
	Let $k=R_2/R_1$ and $z_1\in\{1,-1\}$. Suppose that $(k,t)$ is in the outer region as in Definition \ref{def:regions}. Then the characteristic polynomial of
	\begin{equation}\label{eq:omegaM2}
		R_2\Omega_{R_1,R_2}(z_1,-1)^{-1}M_{H_{t,R_1,R_2}}(z_1,-1)
	\end{equation}
	has four roots \[\Big\{\lambda,-\lambda,\mu,-\mu\Big\}\subset\Qp[\ii],\] such that
	\[\ord_p(\lambda)=\ord_p(k(2t-1))\]
	and
	\[\ord_p(\mu)=\ord_p\left(\frac{t(t-1)}{2t-1}\right).\]
\end{lemma}

\begin{proof}
	By Lemma \ref{lemma:eigenvectors}, $\lambda$ and $\mu$ are the roots of
	\[P(x)=x^2+\ii(k(1-2t)+tz_1)x+kt(1-t)z_1.\]
	The discriminant of this equation is
	\[\Delta=-(k(1-2t)+tz_1)^2-4kt(1-t)z_1.\]
	The order of $k(1-2t)+tz_1$ is that of $k(2t-1)$, because this has negative order and $t$ has nonnegative order. Hence, the order of $(k(1-2t)+tz_1)^2$ is that of $k^2(2t-1)^2$. The order of $kt(1-t)$ is at least that of $k$, which is bigger, because we are in the outer region.
	Hence, $\ord_p(\Delta)=2\ord_p(k(2t-1))$ and it has a square root in $\Qp[\ii]$ of order $\ord_p(k(2t-1))$ with the same leading digit as $\ii k(2t-1)$. We call $\lambda$ the root of $P(x)$ given by this root, and $\mu$ the other root. Now we have that $\ord_p(\lambda)=\ord_p(k(2t-1))$ and
	\[\ord_p(\mu)=\ord_p(kt(t-1)z_1)-\ord_p(k(2t-1))=\ord_p\left(\frac{t(t-1)}{2t-1}\right).\qedhere\]
\end{proof}

See Figure \ref{fig:orders} for a representation of the orders of $\lambda$ and $\mu$ in terms of how close is $t$ to $1/2$.

\begin{figure}
	\includegraphics[scale=0.8]{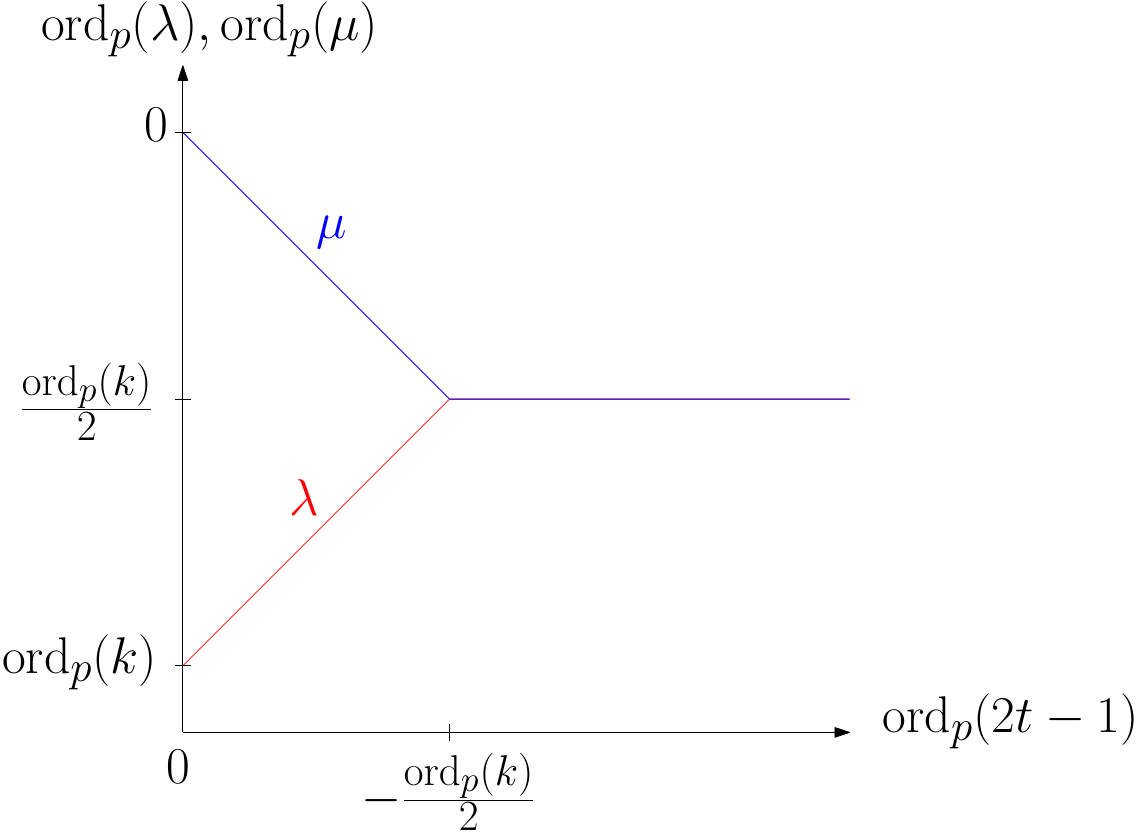}
	\includegraphics[scale=0.8]{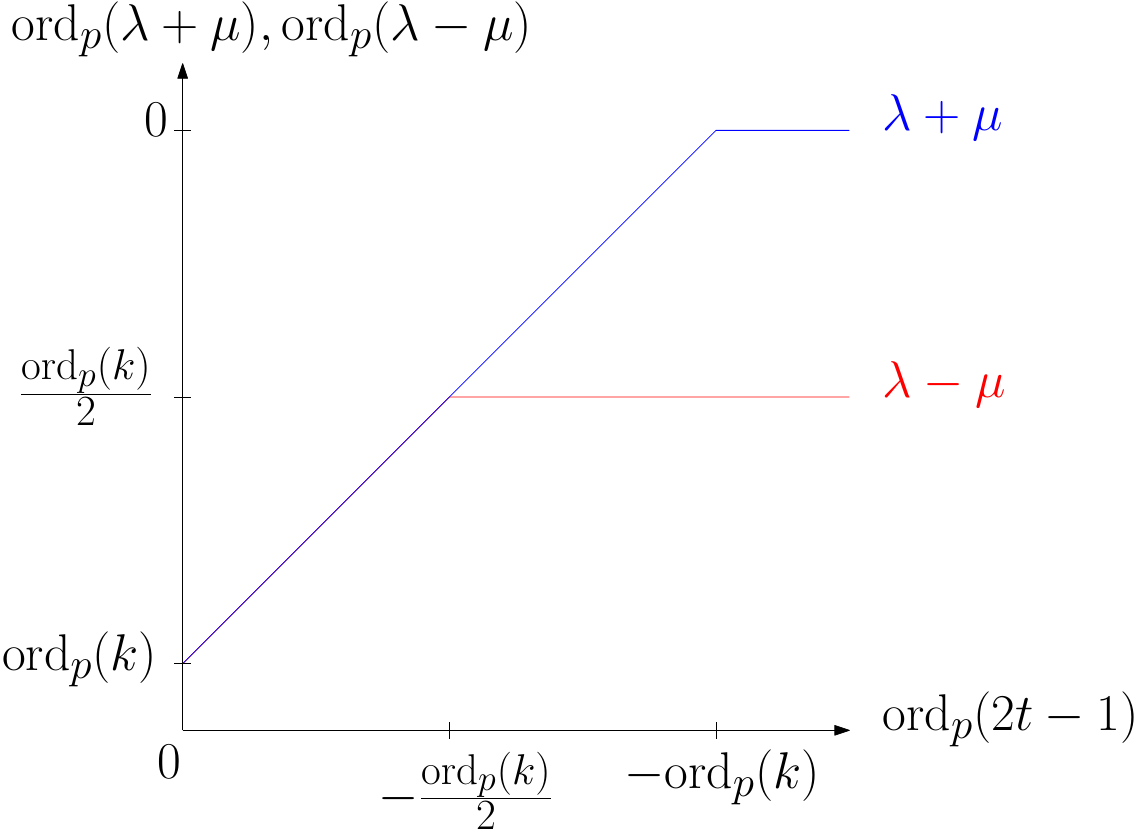}
	\caption{Representation of the orders of the eigenvalues $\lambda$ and $\mu$ (first graphic) and of $\lambda+\mu$ and $\lambda-\mu$ (second graphic) in terms of the order of $2t-1$. The part at the left of $-\ord(k)/2$ is the outer region, and the part at the right is the inner region. At the points where one of the graphs changes slope, the orders may not be correct, because there may be a cancellation between terms in a sum with the same order, resulting in a total with greater order than the terms (this is what happens in the limit region).}
	\label{fig:orders}
\end{figure}

\begin{corollary}\label{cor:normalform2}
	\letpprime. Let $\F$ be the parameter set in \eqref{eq:F} and let $(t,R_1,R_2)\in \F$. Let $F_{t,R_1,R_2}=(J_{t,R_1,R_2},H_{t,R_1,R_2}):\sphere\times\sphere\to\Qp^2$ be the $p$-adic coupled angular momentum system given by expression \eqref{eq:angular}. Let $S,T$ be the points given in \eqref{eq:PQST}, which by Proposition \ref{prop:rank0} are rank zero non-degenerate critical points of $F_{t,R_1,R_2}$. Let $m\in\{S,T\}$. Let $X_p$ be the set given in Definition \ref{def:classes}. If $t\notin\{0,1\}$ and $(k,t)$ is in the outer region of Definition \ref{def:regions}, there exist a matrix $B\in\M_2(\Qp)$, $c_1,c_2\in X_p$ and local coordinates $(x,\xi,y,\eta)$ centered at $m$ such that the $p$-adic analytic symplectic form is given by $\dd x\wedge\dd\xi+\dd y\wedge\dd\eta$ and
	\begin{align*}
		\widetilde{F}_{t,R_1,R_2}(x,\xi,y,\eta) & =B\circ(F_{t,R_1,R_2}(x,\xi,y,\eta)-F_{t,R_1,R_2}(m)) \\
		& =\frac{1}{2}(x^2+c_1\xi^2,y^2+c_2\eta^2)+\ocal((x,\xi,y,\eta)^3).
	\end{align*}
\end{corollary}

\begin{proof}
	That there exist linear symplectic coordinates follows from Lemma \ref{lemma:eigenvalues2} and \cite[Theorem A]{CrePel-williamson}, because $\lambda^2,\mu^2\in\Qp$ corresponds to class (1) in the Williamson classification. The existence of local symplectic coordinates follows from Lemma \ref{lemma:darboux}.
\end{proof}

\begin{lemma}\label{lemma:eigenvectors-lambda2}
	\letpprime. Let $\F$ be the parameter set in \eqref{eq:F} and let $(t,R_1,R_2)\in \F$. Let $F_{t,R_1,R_2}=(J_{t,R_1,R_2},H_{t,R_1,R_2}):\sphere\times\sphere\to\Qp^2$ be the $p$-adic coupled angular momentum system given by expression \eqref{eq:angular}. Let $S,T$ be the points defined in \eqref{eq:PQST}, which by Proposition \ref{prop:rank0} are rank zero non-degenerate critical points of $F_{t,R_1,R_2}$. Then, if $(k,t)$ is in the outer region and $t\notin\{0,1\}$, the value of $c_1$ in Corollary \ref{cor:normalform} is given as follows.
	\begin{enumerate}
		\item If $p\equiv 1\mod 4$ or $p=2$, then $c_1=1$, that is, the component is elliptic (i.e. of the form $(x^2+\xi^2)/2$).
		\item If $p\equiv 3\mod 4$, then $c_1=1$ if $\ord_p(R_1)$ is even and $c_1=p^2$ otherwise.
	\end{enumerate}
\end{lemma}

\begin{proof}
	If $p\equiv 1\mod 4$, then \[\Qp[\ii]=\Qp,\] which implies $\lambda\in\Qp$ and this component is elliptic (or hyperbolic, which is the same in the $p$-adic case).
	
	If $p=2$, a component with eigenvalues in $\Qp[\ii]$ must have a normal form $x^2+c_1\xi^2$ with a $c_1\in X_p$ such that $-c_1$ is the square of an element of $\Qp[\ii]$ outside $\Qp$. Among the elements of the set $X_p$, only $1$ satisfies the condition, hence $c_1=1$.
	
	Now we assume that $p\equiv 3\mod 4$. In this case, there are two elements $c_1\in X_p$ such that $-c_1$ is the square of an element of $\Qp[\ii]$: $1$ and $p^2$. In order to distinguish between them, we use the criterion in \cite[Proposition 4.6]{CrePel-williamson}. Let $A$ be the matrix in \eqref{eq:omegaM2}. Lemma \ref{lemma:eigenvectors} gives us the eigenvector of $A$ corresponding to $\lambda$, which is
	\[(tz_1-\ii\lambda,-\ii tz_1-\lambda,-t,\ii t).\]
	We have that
	\begin{align*}
		v\tr\Omega \bar{v} & =(tz_1-\ii\lambda)R_1z_1(\ii tz_1+\lambda)-(-\ii tz_1-\lambda)R_1z_1(tz_1-\ii\lambda)-\ii R_2t^2-\ii R_2t^2 \\
		& =R_2(2\ii\frac{z_1}{k}(tz_1-\ii\lambda)^2-2\ii t^2).
	\end{align*}
	By Lemma \ref{lemma:eigenvalues2}, $\ord_p(\lambda)=\ord_p(k(2t-1))$ and $\ord_p(v\tr\Omega \bar{v})=\ord_p(kR_2(2t-1)^2)$.
	
	Now we use the formula in \cite[Proposition 4.6]{CrePel-williamson} with $a=b=\ii\lambda/R_2$. We need to check whether
	\[\frac{2a\lambda/R_2}{v\tr\Omega \bar{v}}=\frac{2\ii\lambda^2}{R_2^2v\tr\Omega \bar{v}}\]
	is in
	\begin{equation}\label{eq:eigenvectors-lambda2}
		\DSq(\Qp,-\lambda^2)=\Big\{r^2-\lambda^2s^2:r,s\in\Qp\Big\},
	\end{equation}
	which is the same as $\DSq(\Qp,1)$. By \cite[Proposition 4.7]{CrePel-williamson}, it is in \eqref{eq:eigenvectors-lambda2} if and only if $\ord_p(kR_2(2t-1)^2)$ is even, that is, exactly when $\ord_p(R_1)$ is even. Hence, in this case we may take $c_1=1$, and if $\ord_p(R_1)$ is odd, we take $c_1=p^2$ instead.
\end{proof}

\begin{lemma}\label{lemma:eigenvectors-mu2}
	\letpprime. Let $\F$ be the parameter set in \eqref{eq:F} and let $(t,R_1,R_2)\in \F$. Let $F_{t,R_1,R_2}=(J_{t,R_1,R_2},H_{t,R_1,R_2}):\sphere\times\sphere\to\Qp^2$ be the $p$-adic coupled angular momentum system given by expression \eqref{eq:angular}. Let $S,T$ be the points defined in \eqref{eq:PQST}, which by Proposition \ref{prop:rank0} are rank zero non-degenerate critical points of $F_{t,R_1,R_2}$. Then, if $(k,t)$ is in the outer region and $t\notin\{0,1\}$, the value of $c_2$ in Corollary \ref{cor:normalform} is given as follows.
	\begin{enumerate}
		\item If $p\equiv 1\mod 4$ or $p=2$, then $c_2=1$, that is, the component is elliptic (i.e. of the form $(y^2+\eta^2)/2$).
		\item If $p\equiv 3\mod 4$, then $c_2=1$ if $\ord_p(R_2)$ is even and $c_2=p^2$ otherwise.
	\end{enumerate}
\end{lemma}

\begin{proof}
	If $p\equiv 1\mod 4$, then \[\Qp[\ii]=\Qp\] and $\mu\in\Qp$, and the component is elliptic.
	
	If $p=2$, then we are in the same situation as with the other component: among the values in $X_p$ (recall that $X_p$ was defined in Definition \ref{def:classes}), only $c_2=1$ satisfies that $-c_2$ is the square of an element of $\Qp[\ii]$ outside $\Qp$.
	
	If $p\equiv 3\mod 4$, this condition is satisfied by $c_2=1$ and $c_2=p^2$, so we need the eigenvector criterion. Let $A$ be the $p$-adic matrix in \eqref{eq:omegaM}. Lemma \ref{lemma:eigenvectors} gives us the eigenvector of $A$ corresponding to $\lambda$, which is
	\[(tz_1-\ii\mu,-\ii tz_1-\mu,-t,\ii t).\]
	We have that
	\begin{align*}
		v\tr\Omega \bar{v} & =(tz_1-\ii\mu)R_1z_1(\ii tz_1+\mu)-(-\ii tz_1-\mu)R_1z_1(tz_1-\ii\mu)-\ii R_2t^2-\ii R_2t^2 \\
		& =R_2(2\ii\frac{z_1}{k}(tz_1-\ii\mu)^2-2\ii t^2).
	\end{align*}
	By Lemma \ref{lemma:eigenvalues2}, $\ord_p(\mu)=\ord_p(t(t-1)/(2t-1))$, hence the first term has order greater or equal than $2\ord_p(t)-\ord_p(2t-1)^2-\ord_p(k)>2\ord_p(t)$ (because we are in the outer region), and the order of $v\tr\Omega \bar{v}$ is $\ord_p(R_2)+2\ord_p(t)$, which has the same parity as $\ord_p(R_2)$.
	
	Now we use the formula in \cite[Proposition 4.6]{CrePel-williamson} with $a=b=\ii\mu/R_2$. The result is
	\[\frac{2a\mu/R_2}{v\tr\Omega \bar{v}}=\frac{2\ii\mu^2}{R_2^2v\tr\Omega \bar{v}}.\]
	This is in \[\DSq(\Qp,-\mu^2)=\DSq(\Qp,1),\] when $\ord_p(R_2)$ is even, hence $c_2=1$ is valid in that case, and otherwise we take $c_2=p^2$.
\end{proof}

\section{Normal forms of $F_{t,R_1,R_2}$ at $S$ and $T$ in the inner region}\label{sec:ST-inner}

\begin{figure}
	\includegraphics{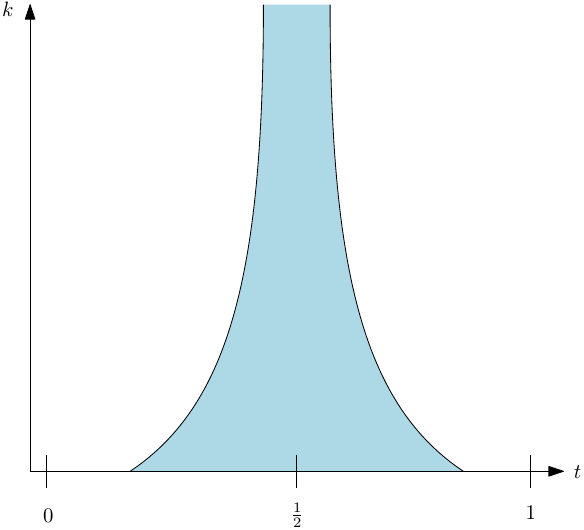}
	\caption{The inner region, in terms of the parameters $k$ and $t$.}
	\label{fig:inner}
\end{figure}

Now we study the normal forms at $S$ and $T$ in the cases where $p\ne 2$ and $\ord_p(k(2t-1)^2)>0$ (see Figure \ref{fig:inner}). We first note that in this case $\ord_p(2t-1)>0$, and this implies $\ord_p(t)=\ord_p(t-1)=0$. This result is the analog of Lemma \ref{lemma:eigenvalues2}.

\begin{lemma}\label{lemma:eigenvalues3}
	\letpprime. Let $\F$ be the parameter set in \eqref{eq:F} and let $(t,R_1,R_2)\in \F$. Let $M_{J_{t,R_1,R_2}}(z_1,z_2)$ and $M_{H_{t,R_1,R_2}}(z_1,z_2)$ be the two matrices in \eqref{eq:MJ} and \eqref{eq:MH} respectively. Let $\Omega_{R_1,R_2}(z_1,z_2)$ be the matrix of $\omega=R_1\omega_1\oplus R_2\omega_2$ as in equation \eqref{eq:omega}.
	Let $k=R_2/R_1$ and $z_1\in\{1,-1\}$. Suppose that $(k,t)$ is in the inner region as in Definition \ref{def:regions}. Then the characteristic polynomial of \eqref{eq:omegaM2}
	has four roots \[\Big\{\lambda,-\lambda,\mu,-\mu\Big\}\subset\Qp[\ii,\sqrt{k}],\] such that
	\begin{equation}\label{eq:lambda3}
		\ord_p(\lambda)=\ord_p(\mu)=\frac{\ord_p(k)}{2},
	\end{equation}
	$\lambda$ has the same leading digit as $\sqrt{kz_1}/2$, and $\mu$ has the opposite. The roots are pairwise distinct if and only if $k(1-2t)+tz_1\ne 0$.
\end{lemma}

\begin{proof}
	By Lemma \ref{lemma:eigenvectors}, $\lambda$ and $\mu$ are the roots of
	\[P(x)=x^2+\ii(k(1-2t)+tz_1)x+kt(t-1)z_1.\]
	The discriminant of this equation is
	\[\Delta=-(k(1-2t)+tz_1)^2-4kt(t-1)z_1.\]
	The order of $k(1-2t)+tz_1$ is at least $\min\{\ord_p(k(2t-1)),0\}$. Hence, the order of $(k(1-2t)+tz_1)^2$ is at least $\min\{\ord_p(k^2(2t-1)^2),0\}$. The order of $kt(t-1)$ is that of $k$, which is smaller, because we are in the inner region.
	Hence, $\ord_p(\Delta)=\ord_p(k)$. It has a square root in $\Qp[\ii,\sqrt{kt(t-1)}]$ of order $\ord_p(k)/2$. This extension is the same as $\Qp[\ii,\sqrt{k}]$, because $t$ has the same leading digit than $1/2$ and $\sqrt{t(t-1)}$ has that of $\ii/2$. No root of the discriminant cancels with $-\ii(k(1-2t)+tz_1)$, which has higher order, hence both roots of $P(x)$ have order $\ord(k)/2$.
	
	Respecting to their leading digits, they are the same as those of $\sqrt{\Delta}/2$, that is, those of $\sqrt{-4kt(t-1)z_1}/2$, which in turn are the same as those of $\sqrt{kz_1}/2$.
	
	Since $\Delta\ne 0$ and $\lambda\ne 0$, the only case in which there may be repeated roots is when $\lambda=-\mu$. By the expression of $P(x)$, this implies $k(1-2t)+tz_1=0$.
\end{proof}

\begin{lemma}\label{lemma:eigenvectors3}
	\letpprime. Let $\F$ be the parameter set in \eqref{eq:F} and let $(t,R_1,R_2)\in \F$. Let $F_{t,R_1,R_2}=(J_{t,R_1,R_2},H_{t,R_1,R_2}):\sphere\times\sphere\to\Qp^2$ be the $p$-adic coupled angular momentum system given by expression \eqref{eq:angular}. Let $k\in\Qp$ and $t\in \Zp$ such that $(k,t)$ is in the inner region as in Definition \ref{def:regions}. Then the local normal form of $F_{t,R_1,R_2}$ at $S$ and $T$ is as follows:
	\begin{enumerate}
		\item If $p\equiv 1\mod 4$:
		\begin{enumerate}
			\item If $k$ is a square in $\Qp$, the normal form is $(1,1)$ of class (1) (i.e. elliptic-elliptic, or equivalently hyperbolic-hyperbolic).
			\item If $k$ has even order but is not a square, the normal form is $c=c_0$ of class (2).
			\item If $k$ is $p$ times a square, the normal form is $c=p$ of class (2).
			\item If $k$ has odd order but it is not $p$ times a square, the normal form is $c=c_0p$ of class (2).
		\end{enumerate}
		\item If $p\equiv 3\mod 4$:
		\begin{enumerate}
			\item If $kz_1$ is a square in $\Qp$, the normal form is $c=-1$ of class (2) (i.e. focus-focus).
			\item If $kz_1$ has even order but is not a square, the normal form is $(1,1)$ of class (1) if $\ord_p(R_2)$ is even and $(p^2,p^2)$ of class (1) otherwise.
			\item If $kz_1$ is $p$ times a square, the normal form is $(-p,-1,0,1,0)$ of class (3).
			\item If $kz_1$ has odd order but it is not $p$ times a square, the normal form is $(p,-1,0,1,0)$ of class (3).
		\end{enumerate}
	\end{enumerate}
\end{lemma}

\begin{proof}
	Let $M_{J_{t,R_1,R_2}}(z_1,z_2)$ and $M_{H_{t,R_1,R_2}}(z_1,z_2)$ be the two matrices in \eqref{eq:MJ} and \eqref{eq:MH} respectively. Let $\Omega_{R_1,R_2}(z_1,z_2)$ be the matrix of $\omega=R_1\omega_1\oplus R_2\omega_2$ as in equation \eqref{eq:omega}.
	
	We suppose first that $k(1-2t)+tz_1\ne 0$. Then we use Lemma \ref{lemma:eigenvalues3} to decide the linear normal form in terms of the eigenvalues of \eqref{eq:omegaM2}, which are pairwise distinct in this case (the constant factor $R_2$ is always in $\Qp$ and does not affect the type of the eigenvalues). By Lemma \ref{lemma:darboux}, this is also a local normal form.
	\begin{enumerate}
		\item If $p\equiv 1\mod 4$:
		\begin{enumerate}
			\item If $k$ is a square in $\Qp$, the eigenvalues are in $\Qp$ and there is only one possible normal form: $(1,1)$ of class (1).
			\item If $k$ has even order but is not a square, the eigenvalues are of the form $\pm r\pm \sqrt{c_0}s$ for $r,s\in\Qp$, which leaves us with only one possible form: $c=c_0$ of class (2).
			\item If $k$ is $p$ times a square, the eigenvalues are of the form $\pm r\pm \sqrt{p}s$ for $r,s\in\Qp$, which leaves us with only one possible form: $c=p$ of class (2).
			\item If $k$ has odd order but it is not $p$ times a square, the eigenvalues are of the form $\pm r\pm \sqrt{c_0p}s$ for $r,s\in\Qp$, which leaves us with only one possible form: $c=c_0p$ of class (2).
		\end{enumerate}
		\item If $p\equiv 3\mod 4$:
		\begin{enumerate}
			\item If $kz_1$ is a square in $\Qp$, the eigenvalues are of the form $\pm r\pm\ii s$ for $r,s\in\Qp$, and the normal form is $c=-1$ of class (2).
			\item If $kz_1$ has even order but is not a square, the eigenvalues are of the form $\pm \ii r,\pm\ii s$ for $r,s\in\Qp$, which leaves us with three possible normal forms: $(1,1)$, $(1,p^2)$ and $(p^2,p^2)$, all of class (1).
			\item If $kz_1$ is $p$ times a square, the eigenvalues are of the form $\pm \ii r\pm\sqrt{p} s$ for $r,s\in\Qp$, which leaves us with two possible forms: $(-p,-1,0,1,0)$ and $(-p,-1,0,1,1)$ of class (3).
			\item If $kz_1$ has odd order but it is not $p$ times a square, the eigenvalues are of the form $\pm \ii r\pm\ii\sqrt{p} s$ for $r,s\in\Qp$, which leaves us with two possible forms: $(p,-1,0,1,0)$ and $(p,-1,0,1,1)$ of class (3).
		\end{enumerate}
	\end{enumerate}
	
	The next step is to use the eigenvector criterion to finish the classification in the cases where there is more than one possible form, concretely 2b, 2c and 2d.
	
	Lemma \ref{lemma:eigenvectors} gives us the expression of the eigenvectors: the one corresponding to $\lambda$ is $v=(tz_1-\ii\lambda,-\ii tz_1-\lambda,tz_2,-\ii tz_2)$ and the one for $\mu$ is the same but with $\mu$ instead of $\lambda$, and
	\[v\tr\Omega\bar{v}=2\ii R_1z_1(tz_1-\ii\lambda)^2-2\ii R_2t^2,\]
	where $\bar{v}$ is the eigenvector for $-\lambda$, also from Lemma \ref{lemma:eigenvectors}.
	
	In case 2b, we need the criterion of \cite[Proposition 4.6]{CrePel-williamson}, with $a=b=\ii\lambda/R_2$: we need to check whether
	\[\frac{2a\lambda/R_2}{v\tr\Omega\bar{v}}=\frac{2\ii\lambda^2}{R_2^2v\tr\Omega\bar{v}}\]
	is in $\DSq(\Qp,1)$, which reduces to determining whether its order is odd or even. Since $\lambda$ has integer order in this case, this in turn reduces to whether the order of $v\tr\Omega\bar{v}$ is odd or even. To this end, we use that, by Lemma \ref{lemma:eigenvalues3}, the order and leading digit of $\lambda$ are those of $\sqrt{kz_1}/2$, so those of $tz_1-\ii\lambda$ must be those of $-\ii\sqrt{kz_1}/2$. After squaring and multiplying by $2\ii R_1z_1$, it becomes $-\ii kR_1/2=-\ii R_2/2$. The same happens exactly with $-2\ii R_2t^2$, because $t$ is near $1/2$. This means that the two terms in $v\tr\Omega\bar{v}$ do not cancel, and $\ord_p(v\tr\Omega\bar{v})=\ord_p(R_2),$ so $c_1=1$ if this is even and $c_1=R^2$ if it is odd. Using $\mu$ instead of $\lambda$ does not change anything, and the result for $c_2$ is the same.
	
	In case 2c, we need to classify a point which is in one of two forms of class (3), so we need the criterion of \cite[Proposition 5.2]{CrePel-williamson}. We have $\alpha=\sqrt{c}=\ii\sqrt{p}$, $\gamma=\sqrt{t_1+t_2\alpha}=\ii$ and $a=1$. Trying $b=0$ (if it does not work, we know that it is $b=1$), we need to check whether
	\[d=\frac{a\alpha\gamma(b+\alpha)}{v\tr\Omega\bar{v}}=\frac{-\ii p}{v\tr\Omega\bar{v}}\]
	is in $\DSq(\Qp[\ii\sqrt{p}],1)$. By \cite[Proposition 6.3 and Table 7]{CrePel-williamson}, if we write $d$ in the form $r'+s'\alpha$, this is equivalent to having $\ord_p(r')\le\ord_p(s')$, which in turn is equivalent to saying that $d$ has integer order, and this to saying that $v\tr\Omega\bar{v}$ has integer order. The two terms in $v\tr\Omega\bar{v}$ have order $\ord_p(R_2)$ and, by the same argument as in the previous paragraph, they do not cancel. Hence $b=0$ works.
	
	Finally, in case 2d, we use the same criterion as in the previous case and we have $\alpha=\sqrt{p}$, $\gamma=\ii$, $a=1$. Trying $b=0$, we need to check whether
	\[d=\frac{a\alpha\gamma(b+\alpha)}{v\tr\Omega\bar{v}}=\frac{\ii p}{v\tr\Omega\bar{v}}\]
	is in $\DSq(\Qp[\sqrt{p}],1)$. Again by \cite[Proposition 6.3 and Table 7]{CrePel-williamson}, this is equivalent to having integer order, which also holds in this case, and $b=0$ works again.
	
	It is left to solve the case when $k(1-2t)+tz_1=0$. In this case the matrix in \eqref{eq:omegaM2} has repeated eigenvalues, and we cannot use it to determine the Williamson type of the critical point. Instead, we must use a different linear combination of the Hessians of $J$ and $H$.
	
	We take
	\[A=R_2\Omega_{R_1,R_2}\dd^2(H_{t,R_1,R_2}(z_1,-1)-tz_1J_{t,R_1,R_2}(z_1,-1))=\begin{pmatrix}
		0 & -2tz_1 & 0 & -ktz_1 \\
		2tz_1 & 0 & ktz_1 & 0 \\
		0 & t & 0 & 0 \\
		-t & 0 & 0 & 0
	\end{pmatrix}.\]
	Its eigenvalues are
	\[\lambda,-\lambda,\mu,-\mu=\pm\ii t\pm t\sqrt{kz_1-1}.\]
	Since $k$ has negative order, the conclusions from these eigenvalues are the same than in the case $k(1-2t)+tz_1\ne 0$. We need, again, to apply an eigenvector criterion in the cases 2b, 2c and 2d.
	
	The eigenvector is
	\[v=(\lambda,-\ii z_1\lambda,-\ii tz_1,-t)\]
	and
	\[v\tr\Omega\bar{v}=2\lambda R_1z_1(-\ii z_1\lambda)-2\ii R_2 t^2z_1=-2\ii R_2\left(\frac{\lambda^2}{k}+t^2z_1\right).\]
	The first term has the same order and leading digit as $t^2kz_1/k=t^2z_1$, hence there is no cancellation and this has order $0$. As discussed in the cases when $k(1-2t)+tz_1\ne 0$, in the three cases 2b, 2c and 2d, the normal form is determined by the order of $v\tr\Omega\bar{v}$, hence we arrive at the same conclusions.
\end{proof}

\section{Normal forms of $F_{t,R_1,R_2}$ at $S$ and $T$ in the limit region}\label{sec:ST-limit}

\begin{figure}
	\includegraphics{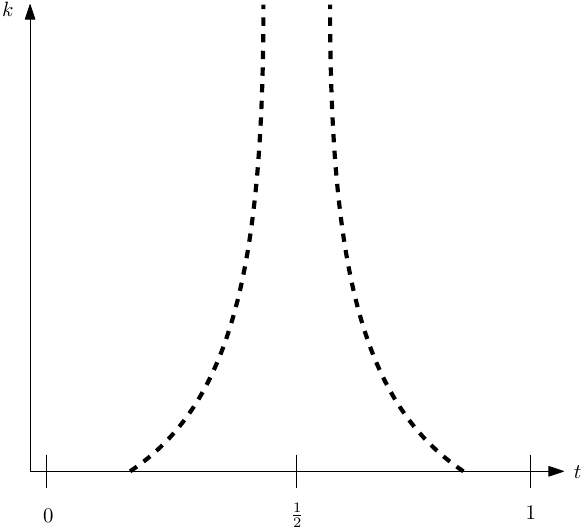}
	\caption{The dotted lines represent the limit region.}
	\label{fig:limit}
\end{figure}

In this section we study the cases in the limit region, when $\ord_p(k(2t-1)^2)=0$ (see Figure \ref{fig:limit}). The results in this case are similar to those in the inner region, but with different formulas. In this case $p\ne 2$, $\ord_p(k)=-2\ord_p(2t-1)$ is even, $\ord_p(2t-1)>0$, $\ord_p(t)=\ord_p(t-1)=0$ and $\ord_p(k(2t-1))<0$. We also have that $\ord_p(R_1R_2)=\ord_p(kR_1^2)$ is even.

For $k,t\in\Qp$ and $z_1\in\{1,-1\}$, we define
\[\Delta(k,t,z_1)=-k^2(1-2t)^2-t^2+2ktz_1.\]

\begin{lemma}\label{lemma:eigenvalues4}
	\letpprime. Let $\F$ be the parameter set in \eqref{eq:F} and let $(t,R_1,R_2)\in \F$. Let $M_{J_{t,R_1,R_2}}(z_1,z_2)$ and $M_{H_{t,R_1,R_2}}(z_1,z_2)$ be the two matrices in \eqref{eq:MJ} and \eqref{eq:MH} respectively. Let $\Omega_{R_1,R_2}(z_1,z_2)$ be the matrix of $\omega=R_1\omega_1\oplus R_2\omega_2$ as in equation \eqref{eq:omega}.
	Let $k=R_2/R_1$ and $z_1\in\{1,-1\}$. Suppose that $(k,t)$ is in the limit region as in Definition \ref{def:regions}. Then the characteristic polynomial of \eqref{eq:omegaM2}
	has four roots \[\Big\{\lambda,-\lambda,\mu,-\mu\Big\}\subset\Qp[\ii,\sqrt{k}],\] such that
	\begin{equation}\label{eq:lambda4}
		\ord_p(\lambda)=\ord_p(\mu)=\frac{\ord_p(k)}{2}.
	\end{equation}
	The roots are pairwise distinct if and only if $\Delta(k,t,z_1)=0$.
\end{lemma}

\begin{proof}
	By Lemma \ref{lemma:eigenvectors}, $\lambda$ and $\mu$ are the roots of
	\[P(x)=x^2+\ii(k(1-2t)+tz_1)x+kt(t-1)z_1.\]
	The discriminant of this equation is precisely $\Delta(k,t,z_1)$. The order of $k(1-2t)+tz_1$ is the same as $k(2t-1)$ and that of $(k(1-2t)+tz_1)^2$ is $\ord_p(k^2(2t-1)^2)=\ord_p(k)$, because we are in the inner region. The order of $kt(t-1)$ is the same. Hence, $\ord_p(\Delta(k,t,z_1))\ge\ord_p(k)$.
	
	Since $\sqrt{\Delta(k,t,z_1)}$ has at least the same order as $\ii(k(1-2t)+tz_1)$, this is at least the order of $\lambda$ and $\mu$. But we also know that $\ord_p(\lambda\mu)=\ord_p(kt(t-1))=\ord_p(k)$, hence the order of $\lambda$ and $\mu$ is exactly $\ord_p(k)/2$.
	
	Two roots may coincide if $\Delta(k,t,z_1)=0$, in which case $\lambda=\mu$, and may also coincide if $\lambda=-\mu$, but this last case implies $\lambda+\mu=k(1-2t)+tz_1=0$, which is not possible because $k(1-2t)$ has negative order.
\end{proof}

\begin{lemma}\label{lemma:eigenvectors4}
	\letpprime. Let $\F$ be the parameter set in \eqref{eq:F} and let $(t,R_1,R_2)\in \F$. Let $F_{t,R_1,R_2}=(J_{t,R_1,R_2},H_{t,R_1,R_2}):\sphere\times\sphere\to\Qp^2$ be the $p$-adic coupled angular momentum system given by expression \eqref{eq:angular}. Let $k=R_2/R_1$ and $z_1\in\{1,-1\}$. Suppose that $(k,t)$ is in the limit region as in Definition \ref{def:regions} and $\Delta(k,t,z_1)\ne 0$. Then the local normal form of $F_{t,R_1,R_2}$ at $S$ and $T$ is as follows:
	\begin{enumerate}
		\item If $p\equiv 1\mod 4$:
		\begin{enumerate}
			\item If $\Delta(k,t,z_1)$ is a square in $\Qp$, the normal form is $(1,1)$ of class (1) (i.e. elliptic-elliptic, or equivalently hyperbolic-hyperbolic).
			\item If $\Delta(k,t,z_1)$ has even order but is not a square, the normal form is $c=c_0$ of class (2).
			\item If $\Delta(k,t,z_1)$ is $p$ times a square, the normal form is $c=p$ of class (2).
			\item If $\Delta(k,t,z_1)$ has odd order but it is not $p$ times a square, the normal form is $c=c_0p$ of class (2).
		\end{enumerate}
		\item If $p\equiv 3\mod 4$:
		\begin{enumerate}
			\item If $\Delta(k,t,z_1)$ is a square in $\Qp$, the normal form is $c=-1$ of class (2) (i.e. focus-focus).
			\item If $\Delta(k,t,z_1)$ has even order but is not a square, the normal form is $(1,1)$ of class (1) if $\ord_p(R_1R_2\Delta(k,t,z_1))$ is a multiple of $4$, and $(p^2,p^2)$ otherwise.
			\item If $\Delta(k,t,z_1)$ is $p$ times a square, the normal form is $(-p,-1,0,1,1)$ of class (3).
			\item If $\Delta(k,t,z_1)$ has odd order but it is not $p$ times a square, the normal form is $(p,-1,0,1,1)$ of class (3).
		\end{enumerate}
	\end{enumerate}
\end{lemma}

\begin{proof}
	These results are similar to Lemma \ref{lemma:eigenvalues3}, but using $\Delta(k,t,z_1)$ instead of $k$. Let $M_{J_{t,R_1,R_2}}(z_1,z_2)$ and $M_{H_{t,R_1,R_2}}(z_1,z_2)$ be the two matrices in \eqref{eq:MJ} and \eqref{eq:MH} respectively. Let $\Omega_{R_1,R_2}(z_1,z_2)$ be the matrix of $\omega=R_1\omega_1\oplus R_2\omega_2$ as in equation \eqref{eq:omega}. We decide the linear normal forms using the eigenvalues of \eqref{eq:omegaM2}, and by Lemma \ref{lemma:darboux} they must also be local normal forms.
	
	The eigenvalues of \eqref{eq:omegaM2} are of the form $\pm\ii r\pm s\sqrt{\Delta(k,t,z_1)}$. The possible forms, taking into account these eigenvalues, are exactly the same that we see in the proof of Lemma \ref{lemma:eigenvalues3}, in each one of the eight cases. The cases which need further refinement are again 2b, 2c and 2d.
	
	In these cases, we have the eigenvectors from Lemma \ref{lemma:eigenvectors} and we need to calculate $v\tr\Omega\bar{v}$, where $v$ is the eigenvector of $\lambda$:
	\begin{align*}
		v\tr\Omega\bar{v} & =2\ii R_1z_1(tz_1-\ii\lambda)^2-2\ii R_2t^2 \\
		& =2\ii R_1z_1\left(tz_1-\ii\frac{-\ii(k(1-2t)+tz_1)+\sqrt{\Delta(k,t,z_1)}}{2}\right)^2-2\ii R_2t^2 \\
		& =2\ii R_1z_1\left(\frac{2tz_1-k(1-2t)-tz_1-\ii\sqrt{\Delta(k,t,z_1)}}{2}\right)^2-2\ii R_2t^2 \\
		& =2\ii R_1z_1\frac{(tz_1-k(1-2t))^2-2\ii(tz_1-k(1-2t))\sqrt{\Delta(k,t,z_1)}-\Delta(k,t,z_1)}{4}-2\ii R_2t^2.
	\end{align*}
	
	We also have that
	\[(tz_1-k(1-2t))^2=t^2-2kt(1-2t)z_1+k^2(1-2t)^2=4kt^2z_1-\Delta(k,t,z_1),\]
	which implies
	\begin{align*}
		v\tr\Omega\bar{v}
		& =2\ii R_1z_1\frac{4kt^2z_1-2\ii(tz_1-k(1-2t))\sqrt{\Delta(k,t,z_1)}-2\Delta(k,t,z_1)}{4}-2\ii R_2t^2 \\
		& = R_1z_1(2\ii kt^2z_1+(tz_1-k(1-2t))\sqrt{\Delta(k,t,z_1)}-\ii\Delta(k,t,z_1))-2\ii R_2t^2 \\
		& =R_1z_1((tz_1-k(1-2t))\sqrt{\Delta(k,t,z_1)}-\ii\Delta(k,t,z_1)) \\
		& =R_1z_1\sqrt{\Delta(k,t,z_1)}(tz_1-k(1-2t)-\ii\sqrt{\Delta(k,t,z_1)}) \\
		& =R_1z_1\sqrt{\Delta(k,t,z_1)}(2tz_1-2\ii\lambda).
	\end{align*}
	Hence
	\begin{align*}
		\ord_p(v\tr\Omega_0\bar{v}) & =\frac{\ord_p(\Delta(k,t,z_1))}{2}+\ord_p(\lambda)+\ord_p(R_1) \\
		& =\frac{\ord_p(\Delta(k,t,z_1))+\ord_p(k)}{2}+\ord_p(R_1) \\
		& =\frac{\ord_p(\Delta(k,t,z_1))+\ord_p(R_1)+\ord_p(R_2)}{2} \\
		& =\frac{\ord_p(R_1R_2\Delta(k,t,z_1))}{2}.
	\end{align*}
	As in the proof of Lemma \ref{lemma:eigenvalues3}, the tests we need to do in cases 2b, 2c and 2d use this order. In case 2b, the normal form is $(c_1,c_2)$ of class (1), where $c_1$ is $1$ if the order of $v\tr\Omega\bar{v}$ is even and $p^2$ if it is odd, and $c_2$ is the same but with $\mu$ instead of $\lambda$. In this case $\Delta(k,t,z_1)$ has even order, and $R_1R_2$ always has even order, hence this order is integer. It is even if and only if $\ord_p(R_1R_2\Delta(k,t,z_1))$ is a multiple of $4$, which means that in this case the form is $(1,1)$ and in the other case (being a multiple of $4$ plus $2$) it is $(p^2,p^2)$.
	
	In case 2c, the form is $(-p,-1,0,1,0)$ if $\ord_p(v\tr\Omega_0\bar{v})$ is integer and $(-p,-1,0,1,1)$ otherwise. Since $\Delta(k,t,z_1)$ has odd order and $R_1R_2$ has even order, $\ord_p(v\tr\Omega_0\bar{v})$ is not an integer and the form is $(-p,-1,0,1,1)$. Case 2d is the same as 2c but with $p$ instead of $-p$ in the normal forms.
\end{proof}

\begin{lemma}\label{lemma:eigenvectors4-degenerate}
	\letpprime. Let $\F$ be the parameter set in \eqref{eq:F} and let $(t,R_1,R_2)\in \F$. Let $F_{t,R_1,R_2}=(J_{t,R_1,R_2},H_{t,R_1,R_2}):\sphere\times\sphere\to\Qp^2$ be the $p$-adic coupled angular momentum system given by expression \eqref{eq:angular}. Let $k=R_2/R_1$ and $z_1\in\{1,-1\}$. Suppose that $(k,t)$ is in the limit region as in Definition \ref{def:regions} and $\Delta(k,t,z_1)=0$. Then $S$ (if $z_1=1$) or $T$ (if $z_1=-1$) is a degenerate critical point of $F_{t,R_1,R_2}$, and its local normal form at $S$ or $T$ is as follows.
	\begin{enumerate}
		\item If $p\equiv 1\mod 4$, it is the form \emph{R3} in \cite[Theorem 10.6]{CrePel-williamson}, that is,
		\[\left\{
		\begin{aligned}
			g_1(x,\xi,y,\eta) & =x\xi+y\eta; \\
			g_2(x,\xi,y,\eta) & =y\xi.
		\end{aligned}
		\right.\]
		\item If $p\equiv 3\mod 4$, it is the form \emph{I3}$(-1)$ in that list, that is,
		\[\left\{
		\begin{aligned}
			g_1(x,\xi,y,\eta) & =x\eta-y\xi; \\
			g_2(x,\xi,y,\eta) & =\frac{x^2+\xi^2}{2}.
		\end{aligned}
		\right.\]
	\end{enumerate}
\end{lemma}

\begin{proof}
	We can check by computation that the characteristic polynomial of
	\[R_2\Omega_{R_1,R_2}^{-1}\dd^2 (cJ_{t,R_1,R_2}+H_{t,R_1,R_2})=\begin{pmatrix}
		0 & c-k(2t-1) & 0 & -ktz_1 \\
		k(2t-1)-c & 0 & ktz_1 & 0 \\
		0 & t & 0 & tz_1+c \\
		-t & 0 & -tz_1-c & 0
	\end{pmatrix}\]
	is the product of
	\[\lambda^2+\ii\lambda(2c+tz_1+k(1-2t))+(tz_1+c)(k(2t-1)-c)-kt^2z_1\]
	and
	\[\lambda^2-\ii\lambda(2c+tz_1+k(1-2t))+(tz_1+c)(k(2t-1)-c)-kt^2z_1.\]
	If $\Delta(k,t,z_1)=0$, this is equal to
	\begin{equation}\label{eq:charpoly-degenerate}
		\left(\lambda+\frac{\ii (2c+tz_1+k(1-2t))}{2}\right)^2\left(\lambda-\frac{\ii (2c+tz_1+k(1-2t))}{2}\right)^2.
	\end{equation}
	Hence the characteristic polynomial has multiple roots for any $c$, which means that the critical point $S$ or $T$ is degenerate.
	
	In order to determine its normal form, we have that the eigenvalues of \eqref{eq:charpoly-degenerate} are nonzero and in $\Qp[\ii]$ for any choice of $c$. This means that we are in a non-nilpotent case, and the only linear normal forms having nonzero eigenvalues in any linear combination are R3 and I3$(c)$, for $c\in Y_p$. If $p\equiv 1\mod 4$, $\Qp[\ii]=\Qp$, and the case with eigenvalues in $\Qp$ is R3. If $p\equiv 3\mod 4$, the case with eigenvalues in $\Qp[\ii]$ is I3$(-1)$. By Lemma \ref{lemma:darboux}, these are also the local normal forms.
\end{proof}

\section{Normal forms of $F_{t,R_1,R_2}$ for $t=0$ and $t=1$}\label{sec:01}

In this section we analyze the cases where $t=0$ and $t=1$, which were left out in the previous sections. Actually, the statements in these sections will still hold, even in the cases $t=0$ and $t=1$, but they require different proofs, because the existing proofs only work if the $p$-adic matrix $A$ has all eigenvalues different, and this does not happen if $t=0$ and $t=1$.

\begin{proposition}
	\letpprime. Let $\F$ be the parameter set in \eqref{eq:F} and let $(t,R_1,R_2)\in \F$. Let $F_{t,R_1,R_2}=(J_{t,R_1,R_2},H_{t,R_1,R_2}):\sphere\times\sphere\to\Qp^2$ be the $p$-adic coupled angular momentum system given by expression \eqref{eq:angular}. Let $P,Q,S,T$ be the points given in \eqref{eq:PQST}, which by Proposition \ref{prop:rank0} are rank zero non-degenerate critical points of $F_{t,R_1,R_2}$. Let $m\in\{P,Q,S,T\}$. For $i\in\{1,2\}$, define $c_i=p^2$ if $p\equiv 3\mod 4$ and $R_i$ has odd order, and otherwise $c_i=1$. If $t=0$, there exist a matrix $B\in\M_2(\Qp)$ and local coordinates $(x,\xi,y,\eta)$ centered at $m$ such that the $p$-adic analytic symplectic form is given by $\dd x\wedge\dd\xi+\dd y\wedge\dd\eta$ and
	\begin{align*}
		\widetilde{F}_{t,R_1,R_2}(x,\xi,y,\eta) & =B\circ(F_{t,R_1,R_2}(x,\xi,y,\eta)-F_{t,R_1,R_2}(m)) \\
		& =\frac{1}{2}(x^2+c_1\xi^2,y^2+c_2\eta^2)+\ocal((x,\xi,y,\eta)^3).
	\end{align*}
\end{proposition}

\begin{proof}
	We consider the linear combination
	\[R_2\Omega_{R_1,R_2}^{-1}(M_{J_{0,R_1,R_2}}(z_1,z_2)+M_{H_{0,R_1,R_2}}(z_1,z_2))=\begin{pmatrix}
		0 & 1+k & 0 & 0 \\
		-1-k & 0 & 0 & 0 \\
		0 & 0 & 0 & 1 \\
		0 & 0 & -1 & 0
	\end{pmatrix},\]
	which has four different eigenvalues, $\pm\ii(1+k)$ and $\pm\ii$. If $p\not\equiv 3\mod 4$, this gives directly the elliptic-elliptic linear normal form, which by Lemma \ref{lemma:darboux} is also a local normal form. If $p\equiv 3\mod 4$, we need to calculate eigenvectors.
	
	The eigenvector of the first two eigenvalues is $v=(1,\ii,0,0)$ and the condition in \cite[Proposition 4.6]{CrePel-williamson} evaluates to
	\[\frac{2a\lambda/R_2}{v\tr\Omega \bar{v}}=\frac{2\ii(1+k)^2/R_2^2}{2\ii R_1}=\frac{(1+k)^2}{R_1R_2^2}.\]
	This number has even order if and only if $R_1$ has even order, hence in this case the component has $c_1=1$ and otherwise $c_1=p^2$.
	
	The eigenvector of the last two eigenvalues is $v=(0,0,1,\ii)$, and the condition evaluates to
	\[\frac{2a\lambda/R_2}{v\tr\Omega \bar{v}}=\frac{2\ii/R_2^2}{2\ii R_2}=\frac{1}{R_2^3},\]
	which means that this component has $c_2=1$ if $R_2$ has even order and $c_2=p^2$ otherwise.
\end{proof}

\begin{proposition}
	\letpprime. Let $\F$ be the parameter set in \eqref{eq:F} and let $(t,R_1,R_2)\in \F$. Let $F_{t,R_1,R_2}=(J_{t,R_1,R_2},H_{t,R_1,R_2}):\sphere\times\sphere\to\Qp^2$ be the $p$-adic coupled angular momentum system given by expression \eqref{eq:angular}. Let $P,Q,S,T$ be the points given in \eqref{eq:PQST}, which by Proposition \ref{prop:rank0} are rank zero non-degenerate critical points of $F_{t,R_1,R_2}$. Let $m\in\{P,Q,S,T\}$. For $i\in\{1,2\}$, define $c_i=p^2$ if $p\equiv 3\mod 4$ and $R_i$ has odd order, and otherwise $c_i=1$. If $t=1$, there exist a matrix $B\in\M_2(\Qp)$ and local coordinates $(x,\xi,y,\eta)$ centered at $m$ such that the $p$-adic analytic symplectic form is given by $\dd x\wedge\dd\xi+\dd y\wedge\dd\eta$ and
	\begin{align*}
		\widetilde{F}_{t,R_1,R_2}(x,\xi,y,\eta) & =B\circ(F_{t,R_1,R_2}(x,\xi,y,\eta)-F_{t,R_1,R_2}(m)) \\
		& =\frac{1}{2}(x^2+c_1\xi^2,y^2+\eta^2)+\ocal((x,\xi,y,\eta)^3).
	\end{align*}
\end{proposition}

\begin{proof}
	We consider the linear combination
	\[R_2\Omega_{R_1,R_2}^{-1}(z_1M_{J_{1,R_1,R_2}}(z_1,z_2)-M_{H_{1,R_1,R_2}}(z_1,z_2))=\begin{pmatrix}
		0 & z_1-kz_2 & 0 & kz_1 \\
		kz_2-z_1 & 0 & -kz_1 & 0 \\
		0 & z_2 & 0 & 0 \\
		-z_2 & 0 & 0 & 0
	\end{pmatrix}\]
	which has four different eigenvalues, $\pm\ii k$ and $\pm\ii$. If $p\not\equiv 3\mod 4$, this gives the elliptic-elliptic linear normal form (i.e. $(x^2+\xi^2)/2,(y^2+\eta^2)/2$), which by Lemma \ref{lemma:darboux} is also a local normal form. If $p\equiv 3\mod 4$, we need to calculate eigenvectors.
	
	The eigenvector of the first two eigenvalues of the matrix is
	\[v=\Big(z_1-kz_2,\ii(k-z_1z_2),z_2-\frac{z_1}{k},\ii(\frac{z_1z_2}{k}-1)\Big)\]
	and the condition in \cite[Proposition 4.6]{CrePel-williamson} evaluates to
	\begin{align*}
		\frac{2a\lambda/R_2}{v\tr\Omega \bar{v}} & =\frac{2\ii k^2/R_2^2}{-2\ii R_1z_1(z_1-kz_2)(k-z_1z_2)-2\ii R_2z_2(z_2-z_1/k)(z_1z_2/k-1)} \\
		& =-\frac{k^4/R_2^2}{R_1k^2z_1(z_1-kz_2)(k-z_1z_2)+R_2z_2(kz_2-z_1)(z_1z_2-k)} \\
		& =-\frac{R_1^4R_2^2}{R_2(kz_1+z_2)(z_1-kz_2)(k-z_1z_2)} \\
		& =-\frac{R_1^4R_2}{(kz_1+z_2)z_1(z_1-kz_2)^2}.
	\end{align*}
	This has even order if and only if $R_2/k=R_1$ has even order, hence in this case the component has $c_1=1$ and otherwise $c_1=p^2$.
	
	The eigenvector of the last two eigenvalues of the same matrix is
	\[v=(z_1-kz_2,\ii(1-kz_1z_2),z_2-kz_1,\ii(z_1z_2-k))\]
	and the condition in \cite[Proposition 4.6]{CrePel-williamson} evaluates to
	\begin{align*}
		\frac{2a\lambda/R_2}{v\tr\Omega \bar{v}} & =\frac{2\ii/R_2^2}{-2\ii R_1z_1(z_1-kz_2)(1-kz_1z_2)-2\ii R_2z_2(z_2-kz_1)(z_1z_2-k)} \\
		& =-\frac{1}{R_2^2(R_1(1-kz_1z_2)^2+R_2(1-kz_1z_2)(z_1z_2-k))} \\
		& =-\frac{1}{R_2^2(1-kz_1z_2)^2(R_1+R_2z_1z_2)}
	\end{align*}
	which has even order if and only if $R_2$ has even order. This means that the component corresponding to the last two eigenvalues has $c_2=1$ if $R_2$ has even order and otherwise $c_2=p^2$.
\end{proof}

\section{Normal forms of $F_{t,R_1,R_2}$ at rank $1$ points}\label{sec:rank1}

Now we study the rank $1$ critical points. The first step is to find them.

\begin{proposition}\label{prop:rank1}
	\letpprime. Let $\F$ be the parameter set in \eqref{eq:F} and let $(t,R_1,R_2)\in \F$. Let $F_{t,R_1,R_2}$ be the $p$-adic coupled angular momentum system in \eqref{eq:angular}. Then $F_{t,R_1,R_2}$ has critical points of rank $1$ at the following positions:
	\begin{enumerate-alph}
		\item if $t=0$, the points of the form $(0,0,\pm 1,x_2,y_2,z_2)$ and $(x_1,y_1,z_1,0,0,\pm 1)$;
		\item if $t\ne 0$, the points of the form $(x_1,y_1,z_1,x_2,y_2,z_2)$, where
		\begin{equation}\label{eq:z}
			\left\{
			\begin{aligned}
				z_1 & =\frac{1}{2}\left(\frac{t}{ck}+\frac{ck}{t}-\frac{ckt}{(1-t-c)^2}\right); \\
				z_2 & =\frac{1}{2}\left(\frac{t(1-t-c)}{c^2k^2}-\frac{t}{1-t-c}-\frac{1-t-c}{t}\right),
			\end{aligned}
			\right.
		\end{equation}
		$x_1$ and $y_1$ are such that $x_1^2+y_1^2+z_1^2=1$,
		\[
		\left\{
		\begin{aligned}
			x_2 & =\frac{1-t-c}{ck}x_1; \\
			y_2 & =\frac{1-t-c}{ck}y_1,
		\end{aligned}
		\right.
		\]
		$c\in\Qp$ with $c\ne 0$, and $k=R_2/R_1$;
		\item if $t=1$, the points of the form $(x_1,y_1,z_1,x_1,y_1,z_1)$ and $(x_1,y_1,z_1,-x_1,-y_1,-z_1)$.
	\end{enumerate-alph}
\end{proposition}

\begin{proof}
	Let $(t,R_1,R_2)\in\F$. Suppose that some point $(x_1,y_1,z_1,x_2,y_2,z_2)\in\sphere\times\sphere$ is a critical point of rank $1$. That the rank is $1$ implies that there exist $c,c'\in\Qp$ such that $c\dd J+c'\dd H=0$. If $c'=0$, that means $\dd J=0$ and $\dd z_1=\dd z_2=0$, which is only possible if $x_1=y_1=x_2=y_2=0$, and the point has rank $0$ instead of $1$. Hence, we may assume that $c'=-1$ and $c\dd J=\dd H$.
	
	This expands to
	\[c(R_1\dd z_1+R_2\dd z_2)=\dd H.\]
	If we absorb $R_1$ in the constant $c$, we get
	\begin{align*}
		c(\dd z_1+k\dd z_2) & =\dd H \\
		& =\dd((1-t)z_1+t(x_1x_2+y_1y_2+z_1z_2)) \\
		& =(1-t)\dd z_1+tx_1\dd x_2+tx_2\dd x_1+ty_1\dd y_2+ty_2\dd y_1+tz_1\dd z_2+tz_2\dd z_1.
	\end{align*}
	Rearranging
	\[tx_2\dd x_1+ty_2\dd y_1+(1-t+tz_2-c)\dd z_1+tx_1\dd x_2+ty_1\dd y_2+(tz_1-ck)\dd z_2=0.\]
	The linear dependencies between the forms are $x_1\dd x_1+y_1\dd y_1+z_1\dd z_1=x_2\dd x_2+y_2\dd y_2+z_2\dd z_2=0$, hence the previous equation implies
	\begin{equation}\label{eq:rank1}
		\left\{
		\begin{aligned}
			tx_2 & =c_1x_1; \\
			ty_2 & =c_1y_1; \\
			1-t+tz_2-c & =c_1z_1; \\
			tx_1 & =c_2x_2; \\
			ty_1 & =c_2y_2; \\
			tz_1-ck & =c_2z_2,
		\end{aligned}
		\right.
	\end{equation}
	for some $c_1,c_2\in\Qp$.
	
	\textit{Case 1: $t=0$.} Then we have that $c_1x_1=c_1y_1=c_2x_2=c_2y_2=0$. At least one between $c_1$ and $c_2$ must be $0$, otherwise all the coordinates $x_1,y_1,x_2,y_2$ would be $0$ and the point would have rank $0$.
	
	If $c_1=0$, \eqref{eq:rank1} implies that $1-c=c_1z_1=0$, that is, $c=1$. Then $c_2z_2=-k\ne 0$ and $c_2\ne 0$, hence $x_2=y_2=0$ and the point has the form $(x_1,y_1,z_1,0,0,\pm 1)$.
	
	If $c_2=0$, \eqref{eq:rank1} implies that $-ck=0$ and $c=0$. Then $c_1z_1=1$ and $c_1\ne 0$, hence $x_1=y_1=0$ and the point has the form $(0,0,\pm 1,x_2,y_2,z_2)$.
	
	\textit{Case 2: $t\ne 0$.} First we prove that $t^2=c_1c_2$. Suppose on the contrary that $t^2\ne c_1c_2$. By multiplying the first and fourth equations of \eqref{eq:rank1}, we have $t^2x_1x_2=c_1c_2x_1x_2$ and $x_1x_2=0$. If $x_1=0$, then $tx_2=0$ and $x_2=0$, and if $x_2=0$, then $tx_1=0$ and $x_1=0$, so in both cases we have $x_1=x_2=0$. In the same way, starting from the second and fifth equations, we deduce that $y_1=y_2=0$, and the point has rank $0$, not $1$. Hence we must have $t^2=c_1c_2$. Since $t\ne 0$, this implies $c_1\ne 0$ and $c_2\ne 0$.
	
	Now we multiply the third equation by $t$ and the sixth by $c_1$, and add them:
	\[t(1-t+tz_2-c)+c_1(tz_1-ck)=tc_1z_1+c_1c_2z_2=tc_1z_1+t^2z_2\]
	\begin{equation}\label{eq:rank1-2}
		t(1-t-c)-c_1ck=0
	\end{equation}
	
	\textit{Case 2a: $c=0$.} Then \eqref{eq:rank1-2} simplifies to $t(1-t)=0$. Since $t\ne 0$, we have $t=1$. Hence
	\[(x_2,y_2,z_2)=c_1(x_1,y_1,z_1).\]
	Since both points are in the sphere, $c_1\in\{1,-1\}$.
	
	\textit{Case 2b: $c\ne 0$.} Then \eqref{eq:rank1-2} implies
	\[
	\left\{
	\begin{aligned}
		c_1 & =\frac{t(1-t-c)}{ck}; \\
		c_2 & =\frac{t^2}{c_1}=\frac{ckt}{1-t-c}.
	\end{aligned}
	\right.
	\]
	Substituting in \eqref{eq:rank1},
	\begin{equation}\label{eq:rank1-3}
		\left\{
		\begin{aligned}
			x_2 & =\frac{(1-t-c)x_1}{ck}; \\
			y_2 & =\frac{(1-t-c)y_1}{ck}; \\
			z_2 & =\frac{(1-t-c)(tz_1-ck)}{ckt}.
		\end{aligned}
		\right.
	\end{equation}
	Since $(x_1,y_1,z_1)$ and $(x_2,y_2,z_2)$ are both in the sphere,
	\begin{align*}
		1 & =x^2+y^2+z^2 \\
		& =\frac{(1-t-c)^2}{c^2k^2}\left(x_1^2+y_1^2+\left(z_1-\frac{ck}{t}\right)^2\right) \\
		& =\frac{(1-t-c)^2}{c^2k^2}\left(1-\frac{2z_1ck}{t}+\frac{c^2k^2}{t^2}\right),
	\end{align*}
	\[1-\frac{2z_1ck}{t}+\frac{c^2k^2}{t^2}=\frac{c^2k^2}{(1-t-c)^2},\]
	\[z_1=-\frac{t}{2ck}\left(\frac{c^2k^2}{(1-t-c)^2}-1-\frac{c^2k^2}{t^2}\right)=\frac{1}{2}\left(-\frac{ckt}{(1-t-c)^2}+\frac{t}{ck}+\frac{ck}{t}\right).\]
	Substituting in \eqref{eq:rank1-3},
	\[z_2=\frac{1}{2}\left(\frac{t(1-t-c)}{c^2k^2}-\frac{t}{1-t-c}-\frac{1-t-c}{t}\right),\]
	and we are done.
\end{proof}

\begin{definition}
	\letpprime. Let $f:(\Qp)^3\to\Qp$ be given by
	\[f(c,k,t)=(1-t-c)^2(1-t)^2(1-z_1^2)+((1-t-c)^2z_1+c^2kz_2)^2,\]
	where $z_1$ and $z_2$ are given by \eqref{eq:z}.
\end{definition}

\begin{proposition}\label{prop:rank1-type}
	The Williamson type of the critical points in Proposition \ref{prop:rank1} is as follows:
	\begin{enumerate}
		\item those in part (a) of the form $(0,0,\pm 1,x_2,y_2,z_2)$ and in point (c) are transversally elliptic, that is, their local normal form is
		\[g_1(x,\xi,y,\eta)=x^2+\xi^2,g_2=\eta,\]
		except if $p\equiv 3\mod 4$ and $R_1$ has odd order, in which case
		\[g_1(x,\xi,y,\eta)=x^2+p^2\xi^2,g_2=\eta;\]
		\item those in part (a) of the form $(x_1,y_1,z_1,0,0,\pm 1)$ are transversally elliptic, that is, their local normal form is
		\[g_1(x,\xi,y,\eta)=x^2+\xi^2,g_2=\eta,\]
		except if $p\equiv 3\mod 4$ and $R_2$ has odd order, in which case
		\[g_1(x,\xi,y,\eta)=x^2+p^2\xi^2,g_2=\eta;\]
		\item those in part (b) have the local normal form
		\[g_1(x,\xi,y,\eta)=x^2+c'\xi^2,g_2=\eta,\]
		where $c'\in X_p$ is chosen so that $f(c,k,t)/c'$ is a square in $\Qp$ and it has a square root $r$ such that $rR_1(z_1^2-1)\in\DSq(\Qp,c').$
	\end{enumerate}
\end{proposition}

\begin{proof}
	First note that, by Lemma \ref{lemma:darboux}, the linear normal forms obtained from \cite[Theorem A]{CrePel-williamson} are also local normal forms.
	
	\textit{Case (a1): $t=0$ and the point is $(0,0,\hat{z}_1,\hat{x}_2,\hat{y}_2,\hat{z}_2)$ for $\hat{z}_1\in\{1,-1\}.$}
	
	Since the system is invariant by rotation of both spheres around the $z$-axis, we may assume that $\hat{y}_2=0$ at the critical point. Note that this implies $\hat{x}_2\ne 0$.
	
	In this case $c=0$ and $\dd H=0$. $\dd^2 H$ in the coordinates $(x_1,y_1,y_2,z_2)$ is given by
	\[\begin{pmatrix}
		-\hat{z}_1 & 0 & 0 & 0 \\
		0 & -\hat{z}_1 & 0 & 0 \\
		0 & 0 & 0 & 0 \\
		0 & 0 & 0 & 0
	\end{pmatrix}.\]
	We can change coordinates so that the symplectic form is standard at the critical point and $\dd J=\dd\eta$:
	\[x=x_1,\xi=R_1\hat{z}_1y_1,y=\frac{y_2}{\hat{x}_2},\eta=R_2z_2,\]
	and the Hessian becomes
	\[\begin{pmatrix}
		-\hat{z}_1 & 0 & 0 & 0 \\
		0 & -\frac{\hat{z}_1}{R_1^2} & 0 & 0 \\
		0 & 0 & 0 & 0 \\
		0 & 0 & 0 & 0
	\end{pmatrix}.\]
	The submatrix formed by the first two rows and columns, after multiplying by $\Omega_0^{-1}$, results in
	\[\begin{pmatrix}
		0 & \frac{\hat{z}_1}{R_1^2} \\
		-\hat{z}_1 & 0
	\end{pmatrix},\]
	which has as eigenvalues $\ii/R_1$ and $-\ii/R_1$ and as eigenvectors $(-\hat{z}_1\ii/R_1,1)$ and $(\hat{z}_1\ii/R_1,1)$. The first component of the normal form is $x^2+c'\xi^2$ for some $c'\in X_p$. If $p\not\equiv 3\mod 4$, the only possible value is $c'=1$. If $p\equiv 3\mod 4$, we must have $c'=1$ or $c'=p^2$. The criterion in \cite[Proposition 4.6]{CrePel-williamson}, with $c=1$ and $a=1/R_1$, gives
	\[\frac{2a\lambda}{v\tr\Omega_0\bar{v}}=\frac{2\ii/R_1^2}{-2\hat{z}_1\ii/R_1}=-\frac{\hat{z}_1}{R_1}\in\DSq(\Qp,1),\]
	which is equivalent to saying that $R_1$ has even order. Hence $c'=1$ in that case and otherwise $c'=p^2$.
	
	\textit{Case (a2): $t=0$ and the point is $(\hat{x}_1,\hat{y}_1,\hat{z}_1,0,0,\hat{z}_2)$, for $\hat{z}_2\in\{1,-1\}$.}
	
	Again, since the system is invariant by rotation, we assume that $\hat{y}_1=0$ at the critical point. Note that this implies $\hat{x}_1\ne 0$.
	
	Now $c=1$ and $\dd J/R_1=\dd H$. We take coordinates $(y_1,z_1,x_2,y_2)$. The sum $-\dd^2 J/R_1+\dd^2 H$ is given by
	\[\begin{pmatrix}
		0 & 0 & 0 & 0 \\
		0 & 0 & 0 & 0 \\
		0 & 0 & k\hat{z}_2 & 0 \\
		0 & 0 & 0 & k\hat{z}_2
	\end{pmatrix}.\]
	We change coordinates so that the symplectic form is standard and $\dd J=\dd\eta$:
	\[x=x_2,\xi=R_2\hat{z}_2y_2,y=\frac{y_1}{\hat{x}_1},\eta=R_1z_1,\]
	and the Hessian becomes
	\[\begin{pmatrix}
		k\hat{z}_2 & 0 & 0 & 0 \\
		0 & k\frac{\hat{z}_2}{R_2^2} & 0 & 0 \\
		0 & 0 & 0 & 0 \\
		0 & 0 & 0 & 0
	\end{pmatrix}.\]
	The submatrix formed by the first two rows and columns, after multiplying by $\Omega_0^{-1}$, results in
	\[\begin{pmatrix}
		0 & -k\frac{\hat{z}_2}{R_2^2} \\
		k\hat{z}_2 & 0
	\end{pmatrix},\]
	which has as eigenvalues $\ii k/R_2$ and $-\ii k/R_2$, or what is the same, $\ii/R_1$ and $-\ii/R_1$, with corresponding eigenvectors $(\hat{z}_2\ii/R_2,1)$ and $(-\hat{z}_2\ii/R_2,1)$. The first component of the normal form is $x^2+c'\xi^2$ for some $c'\in X_p$. If $p\not\equiv 3\mod 4$, the only possible value is $c'=1$. If $p\equiv 3\mod 4$, we must have $c'=1$ or $c'=p^2$. The criterion in \cite[Proposition 4.6]{CrePel-williamson}, with $c=1$ and $a=1/R_1$, gives
	\[\frac{2a\lambda}{v\tr\Omega_0\bar{v}}=\frac{2\ii/R_1^2}{2\hat{z}_2\ii/R_2}=\frac{R_2\hat{z}_2}{R_1^2}\in\DSq(\Qp,1),\]
	which is equivalent to saying that $R_2$ has even order. Hence $c'=1$ in that case and otherwise $c'=p^2$.
	
	\textit{Case (b): the critical point is one of those in Proposition \ref{prop:rank1}(b).} Let $(\hat{x}_1,\hat{y}_1,\hat{z}_1,\hat{x}_2,\hat{y}_2,\hat{z}_2)$ be the critical point. Since the system is invariant by rotation, we may assume that $\hat{y}_1=0$, which implies that $\hat{y}_2=0$. In this case, we are not at a pole of any sphere: if for example $\hat{x}_1=0$, this implies $\hat{x}_2=0$, so we are at a pole of both spheres and the point would have rank $0$.
	
	We have at those points that $\dd H=c\dd J$, so we must compute the Hessian of $H-cJ$. If we take $(y_1,z_1,y_2,z_2)$ as local coordinates (we can do that because $\hat{x}_1\ne 0$ and $\hat{x}_2\ne 0$), the Hessian of $J$ is $0$, and that of $H$ takes the form
	\begin{equation}\label{eq:hessian-rank1}
		t\begin{pmatrix}
			-\frac{\hat{x}_2}{\hat{x}_1} & 0 & 1 & 0 \\
			0 & -\frac{\hat{x}_2}{\hat{x}_1^3} & 0 & \frac{\hat{z}_1\hat{z}_2}{\hat{x}_1\hat{x}_2}+1 \\
			1 & 0 & -\frac{\hat{x}_1}{\hat{x}_2} & 0 \\
			0 & \frac{\hat{z}_1\hat{z}_2}{\hat{x}_1\hat{x}_2}+1 & 0 & -\frac{\hat{x}_1}{\hat{x}_2^3}
		\end{pmatrix}.
	\end{equation}
	We change coordinates so that the symplectic form is standard and $\dd J=\dd\eta$:
	\[x=\frac{y_1}{\hat{x}_1}-\frac{y_2}{\hat{x}_2},\xi=R_1z_1,y=\frac{y_2}{\hat{x}_2},\eta=R_1z_1+R_2z_2.\]
	The Hessian in the new coordinates is the result of multiplying
	\[\begin{pmatrix}
		\hat{x}_1 & 0 & 0 & 0 \\
		0 & \frac{1}{R_1} & 0 & -\frac{1}{R_2} \\
		\hat{x}_1 & 0 & \hat{x}_2 & 0 \\
		0 & 0 & 0 & \frac{1}{R_2}
	\end{pmatrix}
	M
	\begin{pmatrix}
		\hat{x}_1 & 0 & \hat{x}_1 & 0 \\
		0 & \frac{1}{R_1} & 0 & 0 \\
		0 & 0 & \hat{x}_2 & 0 \\
		0 & -\frac{1}{R_2} & 0 & \frac{1}{R_2}
	\end{pmatrix},\]
	where $M$ is the matrix from \eqref{eq:hessian-rank1}. We only need to take the first two rows and columns of this product, which are
	\begin{equation}\label{eq:partial-hessian-rank1}
		t\begin{pmatrix}
			-\hat{x}_1\hat{x}_2 & 0 \\
			0 & -\frac{\hat{x}_2}{R_1^2\hat{x}_1^3}-\frac{\hat{x}_1}{R_2^2\hat{x}_2^3}-\frac{2}{R_1R_2}\left(\frac{\hat{z}_1\hat{z}_2}{\hat{x}_1\hat{x}_2}+1\right)
		\end{pmatrix}.
	\end{equation}
	The multiplicative constant $t$ can be ignored, because it will not change the Williamson type of the point. After multiplying by $\Omega_0^{-1}$, the eigenvalues are
	\begin{align*}
		\pm\lambda & =\pm\ii\sqrt{\hat{x}_1\hat{x}_2\left(\frac{\hat{x}_2}{R_1^2\hat{x}_1^3}+\frac{\hat{x}_1}{R_2^2\hat{x}_2^3}+\frac{2}{R_1R_2}\left(\frac{\hat{z}_1\hat{z}_2}{\hat{x}_1\hat{x}_2}+1\right)\right)} \\
		& =\pm\ii\sqrt{\frac{\hat{x}_2^2}{R_1^2\hat{x}_1^2}+\frac{\hat{x}_1^2}{R_2^2\hat{x}_2^2}+\frac{2}{R_1R_2}(\hat{x}_1\hat{x}_2+\hat{z}_1\hat{z}_2)} \\
		& =\pm\ii\sqrt{\frac{\hat{x}_2^2(\hat{x}_1^2+\hat{z}_1^2)}{R_1^2\hat{x}_1^2}+\frac{\hat{x}_1^2(\hat{x}_2^2+\hat{z}_2^2)}{R_2^2\hat{x}_2^2}+\frac{2}{R_1R_2}(\hat{x}_1\hat{x}_2+\hat{z}_1\hat{z}_2)} \\
		& =\pm\ii\sqrt{\frac{\hat{x}_2^2}{R_1^2}+\frac{\hat{x}_1^2}{R_2^2}+\frac{2\hat{x}_1\hat{x}_2}{R_1R_2}+\frac{\hat{x}_2^2\hat{z}_1^2}{R_1^2\hat{x}_1^2}+\frac{\hat{x}_1^2\hat{z}_2^2}{R_2^2\hat{x}_2^2}+\frac{2\hat{z}_1\hat{z}_2}{R_1R_2}} \\
		& =\pm\ii\sqrt{\left(\frac{\hat{x}_2}{R_1}+\frac{\hat{x}_1}{R_2}\right)^2+\left(\frac{\hat{x}_2\hat{z}_1}{R_1\hat{x}_1}+\frac{\hat{x}_1\hat{z}_2}{R_2\hat{x}_2}\right)^2} \\
		& =\pm\ii\sqrt{\left(\frac{(1-t-c)\hat{x}_1}{ckR_1}+\frac{\hat{x}_1}{R_2}\right)^2+\left(\frac{(1-t-c)\hat{z}_1}{ckR_1}+\frac{ck\hat{z}_2}{(1-t-c)R_2}\right)^2} \\
		& =\pm\ii\sqrt{\frac{\hat{x}_1^2}{R_2^2}\left(\frac{1-t-c}{c}+1\right)^2+\left(\frac{(1-t-c)\hat{z}_1}{cR_2}+\frac{ck\hat{z}_2}{(1-t-c)R_2}\right)^2} \\
		& =\pm\ii\sqrt{\frac{1-\hat{z}_1^2}{R_2^2}\left(\frac{1-t}{c}\right)^2+\left(\frac{(1-t-c)^2\hat{z}_1+c^2k\hat{z}_2}{c(1-t-c)R_2}\right)^2} \\
		& =\pm\frac{\ii\sqrt{f(c,k,t)}}{c(1-t-c)R_2}.
	\end{align*}
	According to \cite[Proposition 4.6]{CrePel-williamson}, the first component of the normal form is $x^2+c'\xi^2$, where $c'$ is chosen so that $-\lambda^2/c'$ is a square in $\Qp$. This means that $f(c,k,t)/c'$ must be a square, as we wanted. Let $r$ be its square root. The eigenvector is $(\lambda,-\hat{x}_1\hat{x}_2)$, and we need $a$ such that
	\[a^2=-\frac{\lambda^2}{c'}=\frac{f(c,k,t)}{c^2(1-t-c)^2R_2^2c'}\Longrightarrow a=\frac{r}{c(1-t-c)R_2}.\]
	The eigenvector condition gives
	\[\frac{2a\lambda}{-2\lambda\hat{x}_1\hat{x}_2}=-\frac{ack}{(1-t-c)\hat{x}_1^2}=-\frac{r}{R_1(1-t-c)^2(1-\hat{z}_1^2)}\in\DSq(\Qp,c'),\]
	which is equivalent to
	\[rR_1(\hat{z}_1^2-1)\in\DSq(\Qp,c'),\]
	as we wanted.
	
	\textit{Case (c): the critical point is one of those in Proposition \ref{prop:rank1}(c).} In this case, $t=1$, $(\hat{x}_2,\hat{y}_2,\hat{z}_2)=c_1(\hat{x}_1,\hat{y}_1,\hat{z}_1)$ for $c_1\in\{1,-1\}$, and $\dd H=0$. The argument continues in the same way as in the previous case: we assume that $\hat{y}_1=\hat{y}_2=0$, the Hessian of $H$ at the critical point is \eqref{eq:hessian-rank1}, and the first two rows and columns of the Hessian in symplectic coordinates are \eqref{eq:partial-hessian-rank1}. We also have that
	\[\pm\lambda=\pm\ii\sqrt{\left(\frac{\hat{x}_2}{R_1}+\frac{\hat{x}_1}{R_2}\right)^2+\left(\frac{\hat{x}_2\hat{z}_1}{R_1\hat{x}_1}+\frac{\hat{x}_1\hat{z}_2}{R_2\hat{x}_2}\right)^2},\]
	but in this case the expression becomes
	\begin{align*}
		\pm\lambda & =\pm\ii\sqrt{\frac{\hat{x}_1^2}{R_2^2}(c_1k+1)^2+\frac{\hat{z}_1^2}{R_2^2}(c_1k+1)^2} \\
		& =\pm\frac{\ii(c_1k+1)}{R_2}\sqrt{\hat{x}_1^2+\hat{z}_1^2} \\
		& =\pm\frac{\ii(c_1k+1)}{R_2}.
	\end{align*}
	This implies that the first component of the normal form is $x^2+c'\xi^2$, where $c'$ must be $1$ except if $p\equiv 3\mod 4$, in which case $c'\in\{1,p^2\}$. The eigenvectors have the same form as in case 2: $(\lambda,-\hat{x}_1\hat{x}_2)$. In order to have $c'=1$, we need $a$ such that
	\[a^2=-\lambda^2\Longrightarrow a=\frac{c_1k+1}{R_2}.\]
	The eigenvector condition gives
	\[\frac{2a\lambda}{-2\lambda\hat{x}_1\hat{x}_2}=-\frac{a}{c_1\hat{x}_1^2}=-\frac{c_1k+1}{c_1\hat{x}_1^2R_2}\in\DSq(\Qp,1).\]
	This happens if and only if $k/R_2$ has even order, that is, if and only if $R_1$ has even order. Hence, $c'=1$ in this case and $c'=p^2$ otherwise.
\end{proof}

To sum up, the critical points described at Proposition \ref{prop:rank1}(a) and (c) have two possible normal forms. It is not obvious how many forms can be taken by those in part (b). The answer is that everything that may happen actually happens, as the following result shows.

\begin{proposition}\label{prop:rank1-universal}
	\letpprime. Let $\F$ be the parameter set in \eqref{eq:F}. For each $(t,R_1,R_2)\in \F$, let $F_{t,R_1,R_2}=(J_{t,R_1,R_2},H_{t,R_1,R_2}):\sphere\times\sphere\to\Qp^2$ be the $p$-adic coupled angular momentum system given by expression \eqref{eq:angular}. Let $(x^2+c'\xi^2,\eta)$ be any $p$-adic local normal form of a rank $1$ critical point, where $c'\in X_p$. Then there exist $c\ne 0$ and $(t,R_1,R_2)\in\F$ such that the rank $1$ critical point of $F_{t,R_1,R_2}$ described at Proposition \ref{prop:rank1}(b) for that concrete value of $c$ has $(x^2+c'\xi^2,\eta)$ as a local normal form.
\end{proposition}

\begin{proof}
	We start by taking $c$ with order $3$ or $4$, $t$ with order $2$, and $k$ with very big absolute value (that is, very low order). Then we have that
	\[z_1\approx \frac{1}{2}\left(\frac{ck}{t}-ckt\right)\approx\frac{ck}{2t}\]
	and
	\[z_2\approx -\frac{1}{2}\left(\frac{t}{1-t-c}+\frac{1-t-c}{t}\right)\approx -\frac{1}{2t},\]
	where $\approx$ means that both sides have the same order and two leading digits. In particular, $z_1$ has very low order and $z_2$ has order $-2$ or $-3$. Now
	\begin{align*}
		f(c,k,t) & =(1-t-c)^2(1-t)^2(1-z_1^2)+((1-t-c)^2z_1+c^2kz_2)^2 \\
		& \approx -z_1^2+(z_1+c^2kz_2)^2 \\
		& = c^2kz_2(2z_1+c^2kz_2) \\
		& \approx -\frac{c^2k}{2t}\left(\frac{ck}{t}-\frac{c^2k}{2t}\right) \\
		& = \frac{c^3k^2}{2t^2}\left(1-\frac{c}{2}\right) \\
		& \approx \frac{c^3k^2}{2t^2}.
	\end{align*}
	The elements of $\Qp$ with the same order and two leading digits are in the same class modulo squares. In particular, we want the class of $c'$ to be the same as that of $f(c,k,t)$, which in turn is the same as $c^3k^2/2t^2$ or $c/2$. The values of $c$ of order $3$ and $4$ cover all possible classes modulo squares, independently of $p$. Hence, the same happens with $c/2$ and with $c'$.
	
	It is left to show that all values of $c'\in X_p$ which are in a fixed class modulo squares can appear in the normal form. There are at most two values $c'\in X_p$ in a fixed class. Hence, both appearing in the normal form is equivalent to saying that, for a fixed $c'$ in that class, $rR_1(1-z_1^2)$ may or may not be in $\DSq(\Qp,c')$. This happens because $r$ and $z_1$ only depends on $c$, $k$ and $t$, and we can vary $R_1$ and $R_2$ arbitrarily while keeping $k$ constant.
\end{proof}

\section{Proof of Theorems \ref{thm:new-main}, \ref{thm:number}, \ref{thm:main}, \ref{thm:number1} and \ref{thm:main1}}\label{sec:proofs}

\subsection*{Proof of Theorem \ref{thm:main}}

Point (i) follows from Corollary \ref{cor:normalform} and Lemmas \ref{lemma:eigenvectors-lambda} and \ref{lemma:eigenvectors-mu}. Point (ii) follows from Corollary \ref{cor:normalform2} and Lemmas \ref{lemma:eigenvectors-lambda2} and \ref{lemma:eigenvectors-mu2} if the parameters $(t,R_1,R_2)$ are in the outer region, Lemma \ref{lemma:eigenvectors3} if they are in the inner region, Lemma \ref{lemma:eigenvectors4} if they are in the limit region and the critical point is non-degenerate, and Lemma \ref{lemma:eigenvectors4-degenerate} if the critical point is degenerate.

\subsection*{Proof of Theorem \ref{thm:number}}

The description of the critical points follows from Proposition \ref{prop:rank0}, and the number of normal forms from Theorem \ref{thm:main}.

\subsection*{Proof of Theorem \ref{thm:main1}}

This follows from Propositions \ref{prop:rank1-type} and \ref{prop:rank1-universal}.

\subsection*{Proof of Theorem \ref{thm:number1}}

The description of the critical points follows from Proposition \ref{prop:rank1}, and the number of normal forms from Theorem \ref{thm:main1}.

\subsection*{Proof of Theorem \ref{thm:new-main}}

The part about the rank $0$ points follows from Theorem \ref{thm:number}, and the part about the rank $1$ points from Theorem \ref{thm:number1}.

\section{Final remarks}\label{sec:final}

\begin{remark}[Computational aspects of this paper]
	The computations we carry out in order to prove Theorems \ref{thm:number}, \ref{thm:main}, \ref{thm:number1}, \ref{thm:main1} are quite tedious, but they would be much more complicated without having the aforementioned general theory in \cite{CrePel-williamson} available. This was not the case for the $p$-adic Jaynes-Cummings model \cite{CrePel-JC}, for which we carried out all computations by hand.
\end{remark}

\begin{remark}[Works on $p$-adic symplectic geometry]\label{rem:works}
	The idea of developing symplectic geometry of integrable systems over the $p$-adic field was formulated in \cite[Section 7]{PVW}. This paper is the sixth of a sequence \cite{PVW,CrePel-JC,CrePel-williamson,CrePel-nonsqueezing,CrePel-Darboux} initiated by Pelayo, Voevodsky and Warren about ten years ago, and it gives a detailed study of the $p$-adic coupled angular momentum, which is the second example of $p$-adic integrable system which we have worked out in detail after the $p$-adic Jaynes-Cummings system \cite{CrePel-JC}. The goal of these works has been to develop a $p$-adic analog of symplectic geometry and of the theory of integrable systems (see Eliashberg \cite{Eliashberg}, Pelayo \cite{Pelayo-hamiltonian,Pelayo-TopApp2023}, Schlenk \cite{Schlenk} and Weinstein \cite{Weinstein-symplectic} for surveys on different aspects of symplectic geometry), with an eye on eventually implementing it using proof assistants in the context of homotopy type theory and Voevodsky's Univalent Foundations \cite{APW,PelWar,PelWar2}. Our goal in this paper is to symplectically classify the critical points of $p$-adic coupled angular momentum. Recent advances in this direction include Gromov's Nonsqueezing Theorem \cite{CrePel-nonsqueezing}, the Weierstrass-Williamson Classification \cite{CrePel-williamson}, the Jaynes-Cummings Model \cite{CrePel-JC}, and Darboux's Theorem \cite{CrePel-Darboux}. This paper is a natural continuation of \cite{CrePel-JC}, but the structure of the normal forms is much more difficult to analyze in the preset paper.
\end{remark}

\appendix
\section{The $p$-adic numbers}\label{sec:appendix}

In this appendix we recall the notions we need about $p$-adic numbers and their extensions.

The \emph{field $\Qp$ of $p$-adic numbers} is given by a completion of the field $\Q$ of rational numbers.

\begin{definition}
	\letpprime. The \emph{$p$-adic order} of an integer number $n$ is the number
	\[\ord_p(n)=\max\{k\in\N:p^k|n\}.\]
	This definition extends to all rational numbers as
	\[\ord_p\left(\frac{m}{n}\right)=\ord_p(m)-\ord_p(n).\]
	The \emph{$p$-adic absolute value of $x\in\Q$} is defined as
	\[|x|_p=p^{-\ord_p(x)}.\]
	The field $\Qp$ is defined as the metric completion of $\Q$ with respect to the $p$-adic absolute value.
\end{definition}

\begin{figure}
	\includegraphics{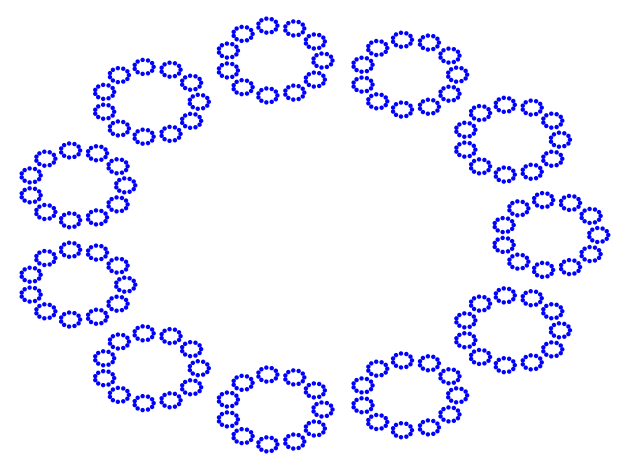}
	\caption{Representation of the $11$-adic numbers. If we consider each point as having radius $1$, then eleven of them have radius $11$, eleven of these balls have radius $11^2$ and all the points in the picture have radius $11^3$.}
	\label{fig:11adics}
\end{figure}

\begin{proposition}
	\letpprime.
	\begin{enumerate}
		\item The $p$-adic order, and with it the $p$-adic absolute value, extend uniquely to $\Qp$.
		\item For $x\in\Qp$, there exist $a_j\in\{0,\ldots,p-1\}$ such that
		\[x=\sum_{j=\ord_p(x)}^\infty a_jp^j\]
		where the sum converges with respect to the $p$-adic absolute value. This expression is known as \emph{$p$-adic expansion of $x$,} and the first coefficient $a_j$ is called \emph{leading digit.}
	\end{enumerate}
\end{proposition}

\begin{proposition}
	\letpprime. The $p$-adic order and the $p$-adic absolute value extend uniquely to the algebraic closure of $\Qp$, but now the order is not necessarily an integer.
\end{proposition}

\begin{definition}
	\letpprime. The metric completion of the algebraic closure of $\Qp$ with respect to the extension of the $p$-adic absolute value to this closure is denoted by $\Cp$, the \emph{complex $p$-adic numbers}.
\end{definition}

\begin{proposition}
	\letpprime\ such that $p\equiv 3\mod 4$. Let $\ii=\sqrt{-1}$. For $x\in\Qp[\ii]$, $\ord_p(x)\in\Z$ and there exist $a_j,b_j\in\{0,\ldots,p-1\}$ such that
	\[x=\sum_{j=\ord_p(x)}^\infty (a_j+\ii b_j)p^j,\]
	where the sum converges with respect to the $p$-adic absolute value extended to $\Qp[\ii]$. The first coefficient $a_j+\ii b_j$ is called \emph{leading digit.}
\end{proposition}

\begin{definition}\label{def:norm}
	\letnpos. \letpprime. We define the \emph{$p$-adic norm of a vector} $v=(v_1,\ldots,v_n)\in(\Cp)^n$ as
	\[\|v\|_p=\max_{1\le i\le n}|v_i|_p.\]
\end{definition}

\section{Review of the real coupled angular momentum}\label{sec:real}

A \emph{symplectic manifold} $(M,\omega)$ is a smooth manifold $M$ endowed with a closed non-degenerate $2$-form $\omega$. If $M$ is $4$-dimensional, an \emph{integrable system} is given by two functions $f_1,f_2$ which are defined on $(M,\omega)$ which Poisson-commute at any point and whose differentials are linearly independent almost everywhere in $M$. We often say that the map itself $F:(M,\omega)\to\R^2$ is the \emph{integrable system.}

\begin{theorem}[{\cite[Lemmas 2.1, 2.4, Propositions 2.5, 2.6, 2.7]{LeFPel}}]\label{thm:real}
	Let $(x_1,y_1,z_1,x_2,y_2,z_2)$ be coordinates on the product $\mathrm{S}^2\times\mathrm{S}^2$ and let $\omega_1$ and $\omega_2$ be the standard symplectic form on the first and second factor of this product, respectively. Endow $\mathrm{S}^2\times\mathrm{S}^2$ with the symplectic form $R_1\omega_1\oplus R_2\omega_2$. Let $t,R_1,R_2\in\R$ such that $0\le t\le 1$ and $R_2>R_1>0$ and define the coupled angular momentum system
	$F_{t,R_1,R_2}:\mathrm{S}^2\times\mathrm{S}^2\to\R^2$ by
	\[F_{t,R_1,R_2}(x_1,y_1,z_1,x_2,y_2,z_2)=\Big(J_{t,R_1,R_2}(x_1,y_1,z_1,x_2,y_2,z_2),H_{t,R_1,R_2}(x_1,y_1,z_1,x_2,y_2,z_2)\Big),\]
	where
	\begin{equation}\label{eq:angular-real}
		\left\{\begin{aligned}
			J_{t,R_1,R_2}(x_1,y_1,z_1,x_2,y_2,z_2) & =R_1z_1+R_2z_2; \\
			H_{t,R_1,R_2}(x_1,y_1,z_1,x_2,y_2,z_2) & =(1-t)z_1+t(x_1x_2+y_1y_2+z_1z_2).
		\end{aligned}\right.
	\end{equation}
	Let $t^+$ and $t^-$ be given by
	\[\frac{R_2}{2R_2+R_1\mp 2\sqrt{R_1R_2}}.\]
	Then the following hold.
	\begin{enumerate}
		\item $F_{t,R_1,R_2}:\mathrm{S}^2\times\mathrm{S}^2\to\R^2$ is an integrable system.
		\item The map $F_{t,R_1,R_2}$ has exactly four critical points $P,Q,S,T$, of rank $0$:
		\[\begin{cases}
			P = (0, 0, 1, 0, 0, 1), & F_{t,R_1,R_2}(P) = (R_1 +R_2, 1), \\
			Q = (0, 0, -1, 0, 0, 1), & F_{t,R_1,R_2}(Q) = (R_2 -R_1, -1), \\
			S = (0, 0, 1, 0, 0, -1), & F_{t,R_1,R_2}(S) = (R_1 -R_2, 1 - 2t), \\
			T = (0, 0, -1, 0, 0, -1), & F_{t,R_1,R_2}(T) = (-(R_1 +R_2), 2t - 1).
		\end{cases}\]
		\item The critical point $S$ of $F_{t,R_1,R_2}$ is non-degenerate of focus-focus type if $t^- < t < t^+$,
		degenerate if $t \in {t^-, t^+}$, and non-degenerate of elliptic-elliptic type otherwise.
		\item For every $t \in [0, 1]$, the critical points $P$, $Q$ and $T$ of $F_{t,R_1,R_2}$ are non-degenerate of elliptic-elliptic type.
		\item When $t \notin \{0, 1\}$, the critical points of rank one of $F_{t,R_1,R_2}$ are the points \[(x_1, y_1, z_1, x_2, y_2, z_2) \in
		\mathrm{S}^2 \times \mathrm{S}^2\] for which there exists $\lambda \in \R \setminus \{0, \frac{1-t}{R_1}\}$ such that
		\[\left\{\begin{aligned}
			(x_2, y_2) & = \frac{1-t-R_1\lambda}{R_2\lambda}(x_1, y_1); \\
			z_1 & =\frac{(t^2+(R_2)^2\lambda^2)(1-t-R_1\lambda)^2-t^2(R_2)^2\lambda^2}{2tR_2\lambda(1-t-R_1\lambda)^2}; \\
			z_2 & =\frac{(t^2-(R_2)^2\lambda^2)(1-t-R_1\lambda)^2-t^2(R_2)^2\lambda^2}{2t(R_2)^2\lambda^2(1-t-R_1\lambda)},
		\end{aligned}\right.\]
		and which are different from the points $P,Q,S,T$ introduced earlier. When $t = 1$, the critical
		points of $F_{t,R_1,R_2}$ of rank one are either the points \[(x_1, y_1, z_1, x_2, y_2, z_2) \in \mathrm{S}^2 \times \mathrm{S}^2 \setminus \{P,Q,S,T\}\] such that
		$(x_2, y_2, z_2) = \pm(x_1, y_1, z_1)$ or those for which there exists $\lambda \ne 0$ such that
		\[\left\{\begin{aligned}
			(x_2,y_2) & =-\frac{R_1}{R_2}(x_1,y_1); \\
			z_1 & =\frac{(R_1)^2-(R_2)^2+(R_1)^2(R_2)^2\lambda^2}{2(R_1)^2R_2\lambda}; \\
			z_2 & =\frac{(R_2)^2-(R_1)^2+(R_1)^2(R_2)^2\lambda^2}{2R_1(R_2)^2\lambda}.
		\end{aligned}\right.\]
		When $t = 0$, the critical points of corank one of $F$ are the points of the form $(0, 0, \pm1, x_2, y_2, z_2)$ with
		$(x_2, y_2, z_2) \ne (0, 0, \pm1)$ or of the form $(x_1, y_1, z_1, 0, 0, \pm1)$ with $(x_1, y_1, z_1) \ne (0, 0, \pm1)$.
	\end{enumerate}
\end{theorem}

\end{document}